\documentclass[11pt,reqno,oneside]{amsart}

\usepackage[utf8]{inputenc}
\usepackage[english]{babel}
\usepackage[a4paper, total={6in, 9in}]{geometry}
\usepackage[abbrev,nobysame,alphabetic]{amsrefs}
\usepackage{comment}
\usepackage{mathrsfs}
\usepackage{bm}
\usepackage{amstext}
\usepackage{amsthm}
\usepackage{amssymb}
\usepackage{cancel}
\usepackage{babel}

\newtheorem{assumption}{Assumptions}[section]
\newtheorem{thm}{Theorem}[section]
\newtheorem*{thm*}{Theorem}
\newtheorem{lem}[thm]{Lemma}
\newtheorem{cor}[thm]{Corollary}
\newtheorem{prop}[thm]{Proposition}

\theoremstyle{definition}
\newtheorem{defn}[thm]{Definition}

\newtheorem{rem}{Remark}[section]
\newtheorem*{rem*}{Remark}

\numberwithin{equation}{section}

\usepackage[svgnames]{xcolor}
\usepackage[bookmarksnumbered=true]{hyperref} 
\hypersetup{
	colorlinks = true,
	linkcolor = Blue,
	anchorcolor = blue,
	citecolor = Green,
	filecolor = blue,
	urlcolor = FireBrick
}
\usepackage{bbm}

\makeatletter
\@namedef{subjclassname@2020}{%
	\textup{2020} Mathematics Subject Classification}
\makeatother

\begin{document}

\title[Semilinear equation fractional Laplacian]{An inverse problem for semilinear equations involving the fractional Laplacian}

\author{Pu-Zhao Kow}
\address{Department of Mathematics and Statistics, University of Jyv\"{a}skyl\"{a},
Jyv\"{a}skyl\"{a}, Finland. }
\email{\href{mailto:pu-zhao.pz.kow@jyu.fi}{pu-zhao.pz.kow@jyu.fi}}

\author{Shiqi Ma}
\address{School of Mathematics, Jilin University, Changchun, China}
\email{\href{mailto:mashiqi@jlu.edu.cn}{mashiqi@jlu.edu.cn}}

\author{Suman Kumar Sahoo}
\address{Department of Mathematics and Statistics, University of Jyv\"{a}skyl\"{a},
Jyv\"{a}skyl\"{a}, Finland. }
\email{\href{mailto:suman.k.sahoo@jyu.fi}{suman.k.sahoo@jyu.fi}}

\subjclass[2020]{35R11, 35R30, 46T20}

\keywords{fractional Laplacian, fractional Calder\'{o}n problem, nonlocal semilinear
equations, fractional diffusion equation, fractional wave equation,
Runge approximation.}

\begin{abstract}
Our work concerns the study of inverse problems of heat and wave equations involving the fractional Laplacian operator with zeroth order nonlinear perturbations. We recover nonlinear terms in the semilinear equations from the knowledge of the fractional  Dirichlet-to-Neumann type map combined with the Runge approximation and the unique continuation property of the fractional Laplacian.
\end{abstract}

\maketitle

\section{Introduction and main results}

We investigate inverse problems for heat and wave equations involving the fractional Laplacian operator with zeroth order nonlinear perturbations.
The study of inverse problems involving the fractional Laplace began with the work \cite{GSU20Calderon} by Ghosh, Salo and Uhlmann.
In \cite{GSU20Calderon}, they proposed and proved a Calder\'on type inverse problem for a linear fractional Laplace operator.
The Calder\'on problem was initiated by Calder\'on in his work \cite{Cal_80} for non-fractional Laplace equations. There is ample amount of literature available on the non-fractional Calder\'on problem and we refer the readers to the survey \cite{Uhl_survey}.
The key tool for studying fractional type of inverse problems is the Runge approximation property, which is a consequence of the fractional unique continuation property (fUCP), i.e.~if $u= (-\Delta)^s u =0$ in certain open set, then $ u=0$ everywhere.
Utilizing these tools, inverse problems involving fractional operators have been greatly investigated by numerous authors in recent years. We refer readers to \cite{GSU_jfa,LL22GlobalUniqueness,LO22GlobalUniquenessSemilinear,Li21GlobalUniquenessSemilinear,Lin20monotonicity,LL22GlobalUniquenessSemilinear1} for some recent works involving inverse problems for fractional semilinear elliptic equations.

Compared to the study of inverse problems involving fractional order operators, the study of inverse problems involving nonlinear terms goes back to Isakov \cite{isakov_nonlinear} and has been under extensive study in the literature.
In \cite{isakov_nonlinear} he studied the nonlinear inverse problems for elliptic and parabolic equations using first order linearization techniques.  In \cite{Mikko_nonlinear_elliptic} the authors successfully implemented {\it higher order linearization techniques} to solve inverse problems for elliptic equations involving power type nonlinearity. In the higher order linearization, the idea is to use product of the solutions of ``\emph{free equation}'' $\Delta u=0$ ($i.e.$ only principal operator, no lower order term is attached). It was observed that using non-linearity as a tool one can solve certain inverse problems which are not available for linear case.
The method was also used to solve several nonlinear inverse problems including partial data \cite{Kru_Uhl_partial,Krupchyk_nonlinear_3,harrach_lin_nonlinear} and Riemannian manifolds \cite{FO_jde,Ali_tony_yisun,Tony_lin_Salo_teemu}. Inverse problems related to more general nonlinearities we refer  \cite{Krupchyk_nonlinear_2,Claudio_partial} and the  references cited there.

The study of inverse problems related to semilinear wave equations with quadratic non-linearity started with the fundamental work \cite{KLU_invention} by Kurylev, Lassas and Uhlmann. In \cite{KLU_invention} the authors used propagation of  non-linear interaction of non smooth plane waves having conormal singularities.  Then in \cite{Ali_lauri}  authors used wave packet (sometimes it is also called quasimode construction) construction to solve certain non-linear hyperbolic inverse problems. This helps to avoid the need to use microlocal analysis techniques. For a comparison between these two methods mentioned above we refer \cite{hintz_uhlmann_zhai}.
Inverse problems for nonlinear parabolic equations have been well studied.
We refer \cite{non_linear_schrodinger,Ali_Yavar_Uhlmann} and the references therein for more results.   

Motivated by the works mentioned above, in this article we consider an inverse problem for nonlinear fractional parabolic equations.
Fractional parabolic equations have applications in random processes \cite{Barlow2009nonlocal}.
We study the fractional type heat equations as well as the fractional type wave equations, and we start with the heat equation first.

Let $n\ge 1$ be a non-negative integer and $0<s<1$. Let $\Omega$ be a bounded Lipschitz
domain in $\mathbb{R}^{n}$ and $\Omega^{e}:=\mathbb{R}^{n}\setminus\overline{\Omega}$.
Let $W$ be any bounded Lipschitz domain in $\Omega^{e}$. Let $u=u(t,x)$
satisfy the following fractional diffusion equation with nonlinear
term $q=q(t,x,z)$: 
\begin{equation}
\begin{cases}
\partial_{t}u(t,x)+(-\Delta)^{s}u(t,x)+q(t,x,u(t,x))=0 & \text{in}\;\;\Omega_{T}\equiv(0,T)\times\Omega,\\
u(t,x)=f(t,x) & \text{in}\;\;\Omega_{T}^{e}\equiv(0,T)\times\Omega^{e},\\
u(0,x)=0 & \forall x \in \Omega,
\end{cases}\label{eq:nonlinear-diffusion-main}
\end{equation}
for certain appropriate exterior data $f=f(t,x)\in\mathcal{C}_{c}^{\infty}(W_{T})$, where $W_T := (0,T) \times W$ and $\mathcal{C}_{c}^{\infty}(\cdot)$ denotes the space of smooth compactly supported functions on their domain of definition.
Here, the fractional Laplacian $(-\Delta)^{s}$ is defined via the
Fourier transform: 
$
\mathscr{F}((-\Delta)^{s}v)(\xi):=|\xi|^{2s}\hat{v}(\xi)$ for all $\xi \in \mathbb{R}^{n},
$
where $\hat{v}=\mathscr{F}v$ is the Fourier transform of distribution
$v$. Given any open sets $V$ and $W$ in $\Omega^{e}$, we define
the DN-map corresponding to \eqref{eq:nonlinear-diffusion-main}
as follows: 
\begin{equation}
\Lambda^{\rm heat}_{q}(f):=(-\Delta)^{s}u \big|_{V_{T}}\quad\text{for all \textquotedblleft sufficiently small\textquotedblright\ }f\in\mathcal{C}_{c}^{\infty}(W_{T}),\label{eq:DN-map}
\end{equation}
where $u$ is the unique solution of \eqref{eq:nonlinear-diffusion-main},
see Proposition~\ref{prop:well-posedness-diffusion-nonlinear}. We now state the assumptions on the coefficient  under which we state and prove our main results.

\begin{assumption}\label{assumption_on_q}
Let $ C^k(\cdot)$ be the space of $k$-times continuously differentiable functions for all integers $k\ge 0$. Assume that the function $q(t,x,z)$ satisfies following conditions.
	\begin{enumerate}
		\renewcommand{\labelenumi}{\theenumi}
		\renewcommand{\theenumi}{(Q.\arabic{enumi})}
	\item \label{itm:q1} For each $(t,x)\in(0,T)\times\Omega$, the mapping $z\mapsto q(t,x,z)$
	is in $\mathcal{C}^{m+1}((-\delta,\delta))$. 
	\item \label{itm:q2} $q(t,x,0)=0$ for all $(t,x)\in\Omega_{T}$. 
	\item \label{itm:q3} There exists a non-decreasing function $\Phi:(-\delta,\delta)\rightarrow\mathbb{R}_{+}$
	such that 
	\[
	\sup_{(t,x)\in\Omega_{T},|z|\le\epsilon}|\partial_{z}q(t,x,z)|\le\Phi(\epsilon)
	\]
	for all $0<\epsilon<\delta$ and $\lim_{\epsilon\rightarrow0}\Phi(\epsilon)=0$. 
	\item \label{itm:q4} Given any $k=2,3,\cdots,m+1$, there exists $M_{k}$ (depending
	on $k$) such that 
	\begin{equation}
	\sup_{(t,x)\in\Omega_{T},|z|\le\delta}|\partial_{z}^{k}q(t,x,z)|\le M_{k}.\label{eq:bound-Mk}
	\end{equation}
	\end{enumerate}
\end{assumption}
With these assumptions on the coefficient, the following is our first main result:
\begin{thm} [Global uniqueness from DN-map] \label{thm:main-diffusion}
	Choose any $n\in\mathbb{N}$ and $0<s<1$.
	Let $\Omega\subset\mathbb{R}^{n}$ be a bounded Lipschitz domain.
	Let $W,V\subset\Omega^{e}$ be any open sets,
	both with Lipschitz boundary, satisfying $\overline{V}\cap\overline{\Omega}=\emptyset$
	and $\overline{W}\cap\overline{\Omega}=\emptyset$. Fix an integer
	$m\ge2$ and a positive number $\delta>0$. Assume that each $q_{j}$ ($j=1,2$) satisfies \ref{itm:q1}-\ref{itm:q4}. 
Then there exists a constant $\tilde{\epsilon}_{0}=\tilde{\epsilon}_{0}(n,s,\Omega,T,\delta)$
	such that, if 
	\[
	\Lambda^{\rm heat}_{q_{1}}(f) = \Lambda^{\rm heat}_{q_{2}}(f) \quad \mbox{for all}\,\, f\in\mathcal{C}_{c}^{\infty}(W_{T})\,\, \mbox{satisfying}\,\, \|f\|_{{\rm ext}}
	\leq \tilde{\epsilon}_{0},
	\]
	where the norm $\|\cdot\|_{{\rm ext}}$ is defined in \eqref{eq:small-f} below,
	then we have 
	\begin{equation} \label{eq:jet-recover}
		\partial_{z}^{k}q_{1}(t,x,0)
		= \partial_{z}^{k}q_{2}(t,x,0) \quad
		\forall (t,x) \in \Omega_{T}, \quad
		k=0,1,2,\cdots,m.
	\end{equation}
	Additionally, if we assume $z \longmapsto  q(t,x,z)$ is analytic for $(t,x)\in \Omega_T$, then we have
	\begin{equation*} \label{eq:recover-potential}
	q_{1}(t,x,z)=q_{2}(t,x,z)\quad \forall (t,x)\in\Omega_{T}, \ \forall z \in I.
	\end{equation*}
\end{thm}
Following the ideas from \cite{GSU_jfa}, one can strengthen above result and  recover the coefficients based on a finite dimensional data set. Our next corollary is related to a single measurement result for linear fractional  Laplace equation, which can be proved  by examining carefully the proof of Theorem \ref{thm:main-diffusion}.

\begin{cor} [Recovery of $m$-jet from $m$-dimensional measurements] \label{cor:diffusion-m-measurement}
Suppose the assumptions in Theorem~{\rm \ref{thm:main-diffusion}} hold. We further assume that for $j=1,2$ 
\[
\partial_{z}^{k} q_{j}(\cdot,0)\in \mathcal{C}^{0}(\overline{\Omega}) \text{ is independent of time variable }t .
\]
Fix any $g_{1},\cdots,g_{m} \in \mathcal{C}_{c}^{\infty}(W_{T})$ such that $g_{1}(t_{0},\cdot),\cdots,g_{m}(t_{0},\cdot) \not\equiv 0$ for some $t_{0} \in (0,T)$. Then 
$
\Lambda^{\rm heat}_{q_{1}}(\epsilon_{1}g_{1} + \cdots + \epsilon_{m}g_{m}) = \Lambda^{\rm heat}_{q_{2}}(\epsilon_{1}g_{1} + \cdots + \epsilon_{m}g_{m})$, 
for all sufficiently small $\epsilon_{j}>0$ ($j=1,\cdots,m$), implies \eqref{eq:jet-recover}.
\end{cor}

In this article, we also take into consideration a nonlinear inverse problem for fractional wave equations in one spatial dimension.
Let $u=u(t,x)$ satisfy
\begin{equation} \label{eq:main-wave}
\begin{cases}
\partial_{t}^{2}u(t,x)+(-\Delta)^{s}u(t,x)+q(t,x,u(t,x))=0 & \text{in}\;\;\Omega_{T},\\
u(t,x)=f(t,x) & \text{in}\;\;\Omega_{T}^{e},\\
u(0,x)=\partial_{t}u(0,x)=0 & \text{for all}\;\;x\in\Omega,
\end{cases}
\end{equation}
for certain appropriate exterior data.
We can define the following
hyperbolic DN-map corresponding to \eqref{eq:main-wave} as follows:
\[
\Lambda_{q}^{{\rm wave}}(f):=(-\Delta)^{s}u \big|_{V_{T}} \quad \text{for all ``sufficiently small'' } f \in \mathcal C_c^{\infty}(W_{T}),
\]
where $u$ is the unique solution of \eqref{eq:main-wave}, see Proposition \ref{prop:well-posedness-wave-nonlinear}
for the well-posedness. The  following result can be proved adapting the similar ideas: 

\begin{thm} [Global uniqueness from DN-map] \label{thm:main-wave}
	Let $n=1$ and $1/2<s<1$.
	Let $\Omega \subset \mathbb{R}$ be a bounded open set, let $W,V \subset \Omega^{e}$ be any open sets satisfying $\overline{V}\cap\overline{\Omega}=\emptyset$ and $\overline{W}\cap\overline{\Omega}=\emptyset$.
	Fix any integer $m \ge 2$ and a positive number $\delta>0$.
	Assume that $q_{j}$ ($j=1,2$) satisfy \ref{itm:q1}--\ref{itm:q4}. Then there exists a constant $\tilde{\epsilon}_{0}=\tilde{\epsilon}_{0}(s,\Omega,T,\delta)$ such that, if $\Lambda_{q_{1}}^{{\rm wave}}(f)=\Lambda_{q_{2}}^{{\rm wave}}(f)$ for all $f\in\mathcal{C}_{c}^{\infty}(W_{T})$ satisfying \eqref{eq:small-f}, then we have \eqref{eq:jet-recover}. Additionally, if we assume $z :\rightarrow q(t,x,z)$ is analytic for $(t,x)\in \Omega_T$ then we have 	\begin{equation*} 
	q_{1}(t,x,z)=q_{2}(t,x,z)\quad \forall (t,x)\in\Omega_{T}, \ \forall z \in I.
	\end{equation*}
\end{thm}

The next corollary is analogous to Corollary \ref{cor:diffusion-m-measurement}.

\begin{cor} [Recovery of $m$-jet from $m$-dimensional measurements] \label{cor:wave-m-measurement}
	Suppose the assumptions in Theorem~{\rm \ref{thm:main-wave}} hold. We further assume that 
		\[
		\partial_{z}^{k} q_{j}(\cdot,0)\in \mathcal{C}^{0}(\overline{\Omega}) \text{ is independent of time variable }t.
		\]
		Fix any $g_{1},\cdots,g_{m} \in \mathcal{C}_{c}^{\infty}(W_{T})$ such that $g_{1}(t_{0},\cdot),\cdots,g_{m}(t_{0},\cdot) \not\equiv 0$ for some $t_{0} \in (0,T)$. If 
	$
\Lambda^{\rm wave}_{q_{1}}(\epsilon_{1}g_{1} + \cdots + \epsilon_{m}g_{m}) = \Lambda^{\rm wave}_{q_{2}}(\epsilon_{1}g_{1} + \cdots + \epsilon_{m}g_{m})
	$
		for all sufficiently small $\epsilon_{j}>0$ ($j=1,\cdots,m$), then we conclude \eqref{eq:jet-recover}.
\end{cor}

There are only a few work available in the literature about the inverse problems for fractional heat equations as well as  fractional wave equations. To motivate our work, we mention several closely related ones.  In \cite{Li21GlobalUniquenessSemilinear}, the author solved certain inverse problems for fractional type heat operators, however the assumptions on the nonlinear term in \cite{Li21GlobalUniquenessSemilinear} are different from ours. Then in  \cite{KLW21CalderonFractionalWave}, the authors studied an inverse problem involving fractional wave equation,
while in \cite{LLL21InverseProblemNonlinearWave}, the authors solved an inverse problem for hyperbolic systems.

The rest of the paper is organized as follows. We discuss the forward problem of the fractional diffusion equation in Section \ref{sec:The-forward-problems}.
We prove a Runge approximation for the fractional diffusion equation in Section \ref{sec:The-Runge-approximation}.
With these tools at hand, Section \ref{sec:inverse-problem-diffusion} is dedicated to the proof of Theorem~\ref{thm:main-diffusion}.
Finally, we investigate Theorem~\ref{thm:main-wave} in Section \ref{sec:Analogous-result-for}.
To make our paper self-contained, we also present the proof of the well-posedness of the linear fractional diffusion equation (Proposition~\ref{prop:well-posed-linear}) in Appendix~\ref{sec:well-posed-linear}.
Then in Appendix \ref{sec:conclusion} we discuss the issue of considering Theorem \ref{thm:main-wave} in one spatial dimension.

\section{The forward problem for the fractional
diffusion equation} \label{sec:The-forward-problems}

In this section, we prove several preliminaries that will be useful in this
work. 

\subsection{Fractional Sobolev spaces}

We use notations for fractional Sobolev spaces as in \cite{KLW21CalderonFractionalWave}.
To make the paper self-contained, we give brief introductions to them.
For $\alpha \in \mathbb R$, denote as $H^{\alpha}(\mathbb{R}^{n})$ the standard $L^{2}$-based fractional Sobolev spaces, which
is defined via Fourier transform \cite{DNPV12FractionalSobolev,Kwa17FractionalEquivalent,Ste16singular}.
For $s\in(0,1)$, in fact
\[
H^{s}(\mathbb{R}^{n})
= \big\{ u \in L^{2}(\mathbb{R}^{n}) \,\big|\, \frac {|u(x)-u(y)|} {|x-y|^{\frac{n}{2}+s}} \in L^{2}(\mathbb{R}^{n}\times\mathbb{R}^{n}) \big\} \quad\text{(as sets)}
\]
with equivalent norm: $
\|u\|_{H^{s}(\mathbb{R}^{n})}^{2}=\|u\|_{L^{2}(\mathbb{R}^{n})}^{2}+[u]_{\dot{H}^{s}(\mathbb{R}^{n})}^{2},
$
where 
\begin{equation}
[u]_{\dot{H}^{s}(\mathbb{R}^{n})}^{2}=\iint_{\mathbb{R}^{n}\times\mathbb{R}^{n}}\frac{|u(x)-u(y)|^{2}}{|x-y|^{n+2s}}\,\mathsf{d}x\,\mathsf{d}y.\label{eq:Gagliardo-seminorm}
\end{equation}
Here, \eqref{eq:Gagliardo-seminorm} is called the Aronszajn-Gagliardo-Slobodeckij
seminorm, see \cite[equation~(2.2)]{DNPV12FractionalSobolev} for reference.

Let $\mathscr{O}$ be any open set in $\mathbb{R}^{n}$, and let $\alpha\in\mathbb{R}$.
We define the following Sobolev spaces:
\begin{align*}
	&H^{\alpha}(\mathscr{O})
:= \, \{ u|_{\mathscr{O}} \big| u\in H^{\alpha}(\mathbb{R}^{n}) \},\,
	\tilde{H}^{\alpha}(\mathscr{O})
	:=  \, \text{closure of }\mathcal{C}_{c}^{\infty}(\mathscr{O})\text{ in }H^{\alpha}(\mathbb{R}^{n}) \\
&	H_{0}^{\alpha}(\mathscr{O})
	:=  \, \text{closure of }\mathcal{C}_{c}^{\infty}(\mathscr{O})\text{ in }H^{\alpha}(\mathscr{O}), \,
	H_{\overline{\mathscr{O}}}^{\alpha}
	:= \, \{ u \in H^{\alpha}(\mathbb{R}^{n}) \,\big|\, {\rm supp}\,(u) \subset \overline{\mathscr{O}} \}.
\end{align*}
The Sobolev space $H^{\alpha}(\mathscr{O})$ is complete under the
quotient norm 
\[
\|u\|_{H^{\alpha}(\mathscr{O})}:=\inf\begin{Bmatrix}\begin{array}{l|l}
\|v\|_{H^{\alpha}(\mathbb{R}^{n})} & v\in H^{\alpha}(\mathbb{R}^{n})\text{ and }v|_{\mathscr{O}}=u\end{array}\end{Bmatrix}.
\]
It is easy to see that $\tilde{H}^{\alpha}(\mathscr{O})\subset H_{0}^{\alpha}(\mathscr{O})$,
and that $H_{\overline{\mathscr{O}}}^{\alpha}$ is a closed subspace
of $H^{\alpha}(\mathbb{R}^{n})$.
If $\Omega$ is a bounded Lipschitz domain, then we also have
following identifications (with equivalent norms): 
\begin{equation*}
	\left\{\begin{aligned}
	\tilde{H}^{\alpha}(\Omega)
	& = H_{\overline{\Omega}}^{\alpha}, \quad
	(H_{\overline{\Omega}}^{\alpha})' = H^{-\alpha}(\Omega)\text{ and }(H^{\alpha}(\Omega))'=H_{\overline{\Omega}}^{-\alpha}\quad \forall \alpha\in\mathbb{R},\\
	H^{s}(\Omega)
	& = H_{\overline{\Omega}}^{s}=H_{0}^{s}(\Omega) \quad \forall -1/2 < s < 1/2,
	\end{aligned}\right.
\end{equation*}
see e.g.~\cite[Section~2A]{GSU20Calderon}, \cite[Chapter~3]{McL00EllipticSystems},
and \cite{Tri02FunctionSpace}. Next  following \cite[Chapter 5]{Eva10PDE}, we define  time dependent fractional Sobolev space for all integers $p\ge 1$ denoted by $L^p((0,T);H^s)$. Then the exterior norm of $f\in \mathcal{C}^{\infty}_c(W_T)$ is given by
	\begin{equation} \label{eq:small-f}
	\|f\|_{{\rm ext}}^{2}
		:= \|f\|_{L^{\infty}(0,T;H^{s}(\mathbb{R}^{n}))\cap L^{\infty}(\mathbb{R}_{T}^{n})}^{2} + \|(-\Delta)^{s}f\|_{L^{2}(\Omega_{T})}^{2}.
	\end{equation}
Moreover, for any measurable set $A \subset \mathbb{R}^{n}$ we use the following notations:
\[
(f,g)_{L^{2}(A)}:=\int_{A}fg\,\mathsf{d}x,\quad(F,G)_{L^{2}(A_{T})}:=\int_{0}^{T}\int_{A}FG\,\mathsf{d}x\,\mathsf{d}t.
\]

\subsection{Well-posedness for the linear equation}

We state the well-posedness of the \emph{linear} fractional diffusion
equation. Let $T>0$, $s\in(0,1)$, and $a=a(t,x)\in L^{\infty}(\Omega_{T})$,
and we consider the following initial-exterior value problem: 
\begin{equation}
\begin{cases}
(\partial_{t}+(-\Delta)^{s}+a)u=F & \text{in}\;\;\Omega_{T},\\
u=f & \text{in}\;\;\Omega_{T}^{e},\\
u=\varphi & \text{in}\;\;\{0\}\times\mathbb{R}^{n},
\end{cases}\label{eq:linear-diffusion1}
\end{equation}
where $f\in\mathcal{C}_{c}^{\infty}(W_{T})$ for some open set with
Lipschitz boundary $W\subset\Omega_{e}$ satisfying $\overline{W}\cap\overline{\Omega}=\emptyset$,
and $\varphi\in\tilde{H}^{0}(\Omega)=\begin{Bmatrix}\begin{array}{l|l}
\varphi\in L^{2}(\mathbb{R}^{n}) & {\rm supp}\,\varphi\subset\overline{\Omega}\end{array}\end{Bmatrix}$.
Setting $v:=u-f$, we then consider the following linear equation
with zero exterior data: 
\begin{equation}
\begin{cases}
(\partial_{t}+(-\Delta)^{s}+a)v=\tilde{F} & \text{in}\;\;\Omega_{T},\\
v=0 & \text{in}\;\;\Omega_{T}^{e},\\
v=\varphi & \text{in}\;\;\{0\}\times\mathbb{R}^{n},
\end{cases}\label{eq:linear-diffusion2}
\end{equation}
where $\tilde{F}=F-(-\Delta)^{s}f$. Now it suffices to study the
well-posedness of \eqref{eq:linear-diffusion2}. 

Define functions $\bm{v}:[0,T]\rightarrow\tilde{H}^{s}(\Omega)$
and $\tilde{\bm{F}}:[0,T]\rightarrow L^{2}(\Omega)$ by 
\begin{equation} \label{eq:bvf-diff2}
	\begin{aligned}
	{} [\bm{v}(t)](x)  :=v(t,x),\quad
	{} [\tilde{\bm{F}}(t)](x) := \tilde F(t,x)\quad\text{for}\;\;(t,x)\in[0,T]\times\mathbb{R}^{n}.
	\end{aligned}
\end{equation}
Let $\langle\cdot,\cdot\rangle$ be the duality pairing on $H^{-s}(\Omega)\oplus\tilde{H}^{s}(\Omega)$.
Multiplying \eqref{eq:linear-diffusion2} by any $\phi\in\tilde{H}^{s}(\Omega)$
gives 
\[
\langle\bm{v}'(t),\phi\rangle+\mathcal{B}[\bm{v},\phi;t]=(\tilde{\bm{F}}(t),\phi)_{L^{2}(\Omega)}\quad\text{for}\;\;0\le t\le T,
\]
where $\mathcal{B}[\bm{v},\phi;t]$ is the bilinear form given by
\[
\mathcal{B}[\bm{v},\phi;t]:=\int_{\mathbb{R}^{n}}(-\Delta)^{s/2}\bm{v}(t)(-\Delta)^{s/2}\phi\,\mathsf{d}x+\int_{\Omega}a(t,\cdot)\bm{v}(t)\phi\,\mathsf{d}x.
\]

\begin{defn}
[Weak solutions] \label{def:weak-solution}We say that $v$ is a
weak solution of \eqref{eq:linear-diffusion2}, if 
\begin{enumerate}
\item [(a)] $v\in L^{2}(0,T;\tilde{H}^{s}(\Omega))$ and $v'\in L^{2}(0,T;H^{-s}(\Omega))$; 
\item [(b)] $\langle\bm{v}'(t),\phi\rangle+\mathcal{B}[\bm{v},\phi;t]=(\tilde{\bm{F}}(t),\phi)_{L^{2}(\Omega)}$
for all $\phi\in\tilde{H}^{s}(\Omega)$ for (almost) all $0\le t\le T$; 
\item [(c)] $\bm{v}(0)=\varphi$,
\end{enumerate}
where $\bm{v}$ and $\tilde{\bm F}$ are defined according to \eqref{eq:bvf-diff2}.
\end{defn}

\begin{prop}[Well-posedness] \label{prop:well-posed-linear}
Given any $n\in\mathbb{N}$
and $0<s<1$. Let $\Omega\subset\mathbb{R}^{n}$ be a bounded Lipschitz
domain in $\mathbb{R}^{n}$. Let $a\in L^{\infty}(\Omega_{T})$. For
any $\tilde{F}\in L^{2}(\Omega_{T})$ and $\varphi\in\tilde{H}^{0}(\Omega)$,
there exists a unique weak solution $v$ of \eqref{eq:linear-diffusion2} and satisfies the following estimate: 
\begin{equation} \label{eq:well-posedness-estimate1a}
\|v\|_{L^{\infty}(0,T;L^{2}(\Omega))}^{2} + \|v\|_{L^{2}(0,T;\tilde{H}^{s}(\Omega))}^{2} + \|\partial_{t}v\|_{L^{2}(0,T;H^{-s}(\Omega))}^{2}
\le C (\|\varphi\|_{L^{2}(\Omega)}^{2}+\|\tilde{F}\|_{L^{2}(\Omega_{T})}^{2} ) 
\end{equation}
for some constant $C=C(n,s,T,\|a\|_{L^{\infty}(\Omega_{T})})$. If
we further assume $\varphi\in\tilde{H}^{s}(\Omega)$, then $v\in L^{\infty}(0,T;\tilde{H}^{s}(\Omega))$
and $\partial_{t}v\in L^{2}(\Omega_{T})$. In this case, the unique
weak solution also satisfies the following estimate: 
\begin{equation} \label{eq:well-posedness-estimate1b}
\|v\|_{L^{\infty}(0,T;\tilde{H}^{s}(\Omega))}^{2}+\|\partial_{t}v\|_{L^{2}(\Omega_{T})}^{2}
\le C (\|\varphi\|_{\tilde{H}^{s}(\Omega)}^{2}+\|\tilde{F}\|_{L^{2}(\Omega_{T})}^{2} )
\end{equation}
for some constant $C=C(n,s,T,\|a\|_{L^{\infty}(\Omega_{T})})$. 
\end{prop}

The proof of Proposition~\ref{prop:well-posed-linear} is analogous to the standard well-posedness proof of the classical diffusion equation.
However, for completeness, we present a proof in Appendix~\ref{sec:well-posed-linear}.

\begin{cor}
\label{cor:well-posed-linear}
Given any $n\in\mathbb{N}$ and $0<s<1$.
Let $\Omega\subset\mathbb{R}^{n}$ be a bounded Lipschitz domain in
$\mathbb{R}^{n}$, and $W\subset\Omega^{e}$ be any open set with
Lipschitz boundary satisfying $\overline{W}\cap\overline{\Omega}=\emptyset$.
Let $a\in L^{\infty}(\Omega_{T})$. Then for any $\tilde{F}\in L^{2}(\Omega_{T})$,
$\varphi\in\tilde{H}^{0}(\Omega)$, and $f\in\mathcal{C}_{c}^{\infty}(W_{T})$,
there exists a unique weak solution $u=v+f$ of \eqref{eq:linear-diffusion1}
satisfying 
\begin{align*}
 & \|u-f\|_{L^{\infty}(0,T;L^{2}(\Omega))}^{2}+\|u-f\|_{L^{2}(0,T;\tilde{H}^{s}(\Omega))}^{2}+\|\partial_{t}(u-f)\|_{L^{2}(0,T;H^{-s}(\Omega))}^{2}\\
 & \le C(\|\varphi\|_{L^{2}(\Omega)}^{2}+\|F-(-\Delta)^{s}f\|_{L^{2}(\Omega_{T})}^{2})
\end{align*}
for some constant $C=C(n,s,T,\|a\|_{L^{\infty}(\Omega_{T})})$. If
we further assume $\varphi\in\tilde{H}^{s}(\Omega)$, then the unique
weak solution $u$ also satisfies the following estimate: 
\begin{equation} \label{eq:diffusion-regularity}
\|u-f\|_{L^{\infty}(0,T;H^{s}(\mathbb{R}^{n}))}^{2}+\|\partial_{t}u\|_{L^{2}(\Omega_{T})}^{2}
\le C (\|\varphi\|_{\tilde{H}^{s}(\Omega)}^{2}+\|F-(-\Delta)^{s}f\|_{L^{2}(\Omega_{T})}^{2})
\end{equation}
for some constant $C=C(n,s,T,\|a\|_{L^{\infty}(\Omega_{T})})$. 
\end{cor}

We skip the proof of Corollary~\ref{cor:well-posed-linear} as it is a straightforward consequence of Proposition~\ref{prop:well-posed-linear}.

\subsection{Maximum principle for the linear equation}

Modifying the ideas in \cite[Proposition~3.1]{LL19GlobalUniqueness}
or \cite[Proposition~4.1]{Ros16FractionalLaplacianSurvey}, we can
obtain the following proposition: 
\begin{prop}
[Maximum principle] \label{prop:maximum-principle}Given any $n\in\mathbb{N}$
and $0<s<1$. Let $\Omega\subset\mathbb{R}^{n}$ be a bounded Lipschitz
domain in $\mathbb{R}^{n}$.
Let $a\in L^{\infty}(\Omega_{T})$.
Suppose that 
$
u\in L^{2}(0,T;H^{s}(\mathbb{R}^{n})) \cap H^{1}(0,T;L^{2}(\Omega))
$
is a weak solution of \eqref{eq:linear-diffusion1}.
If 
$F \geq 0$ in $\Omega_{T}$, $f \geq 0$ in $\Omega_{T}^{e}$, $\varphi \geq 0$ in $\mathbb{R}^{n}$, then $u \geq 0$ in $\Omega_{T}$. 
\end{prop}

\begin{proof}
Let $M$ be a real number which shall be determined later. We define
\begin{equation}
\begin{aligned}
&u_{M}(t,x):=e^{-Mt}u(t,x), \;
a_{M}(t,x):=a(t,x)+M, \;
F_{M}(t,x):=e^{-Mt}F(t,x) \;  \text{in}\;\;\Omega_{T},\\
&f_{M}(t,x):=e^{-Mt}f(t,x) \;\; \text{in}\;\;\Omega_{T}^{e}.
\end{aligned}
\label{eq:reduction-M}
\end{equation}
We see that $u_M$ satisfies
\begin{equation}
\begin{cases}
(\partial_{t}+(-\Delta)^{s}+a_{M})u_{M}=F_{M} & \text{in}\;\;\Omega_{T},\\
u_{M}=f_{M} & \text{in}\;\;\Omega_{T}^{e},\\
u_{M}=\varphi & \text{on}\;\;\{0\}\times\mathbb{R}^{n}.
\end{cases}\label{eq:linear-diffusion4-reduction}
\end{equation}
We choose $M=\|a\|_{L^{\infty}(\Omega_{T})}$, then $a_{M}\ge0$ in $\Omega_{T}$. Next we write $u_{M}=u_{M}^{+}-u_{M}^{-}$, where $u_{M}^{+}=\max\{u_{M},0\}$
and $u_{M}^{-}=\max\{-u_{M},0\}$. Since $u_{M}\in L^{2}(0,T;H^{s}(\mathbb{R}^{n}))\cap H^{1}(0,T;L^{2}(\Omega))$,
then $u_{M}^{\pm}\in L^{2}(0,T;H^{s}(\mathbb{R}^{n})) \cap H^{1}(0,T;L^{2}(\Omega))$ and that
\[
\partial_{t} (u_{M}^{-}) = \begin{cases}
-\partial_{t} u_{M} & \text{in } \{u_{M} < 0 \}, \\
0 & \text{in } \{u_{M} \ge 0 \}.
\end{cases}
\]
Since $u_{M}=f_{M}\ge0$
in $\Omega_{T}^{e}$, hence $u_{M}^{-}=0$ in $\Omega_{T}^{e}$, which
implies $
u_{M}^{-}\in L^{2}(0,T;\tilde{H}^{s}(\Omega)) \cap H^{1}(0,T;L^{2}(\Omega)).$
Testing the first equation of \eqref{eq:linear-diffusion4-reduction} by $u_{M}^{-}$,
we have 
\begin{align*}
0 & \le(\bm{F}_{M}(t),\bm{u}_{M}^{-}(t))_{L^{2}(\Omega)}\quad\text{(because }F_{M}\ge0\text{ and }u_{M}^{-}\ge0\text{ in }\Omega_{T})\nonumber \\
 & =\int_{\Omega}(\partial_{t}\bm{u}_{M}(t))\bm{u}_{M}^{-}(t)\,\mathsf{d}x+\int_{\mathbb{R}^{n}}(-\Delta)^{\frac{s}{2}}\bm{u}_{M}(t)(-\Delta)^{\frac{s}{2}}\bm{u}_{M}^{-}(t)\,\mathsf{d}x+\int_{\Omega}a_M(t,\cdot)\bm{u}_{M}\bm{u}_{M}^{-}\,\mathsf{d}x\nonumber \\
 & =-\frac{\mathsf{d}}{\mathsf{d}t}\big(\frac{1}{2}\int_{\Omega}|\bm{u}_{M}^{-}(t)|^{2}\,\mathsf{d}x\big)+\int_{\mathbb{R}^{n}}(-\Delta)^{\frac{s}{2}}\bm{u}_{M}(t)(-\Delta)^{\frac{s}{2}}\bm{u}_{M}^{-}(t)\,\mathsf{d}x-\int_{\Omega}a_M(t,\cdot)|\bm{u}_{M}^{-}|^{2}\,\mathsf{d}x \label{eq:maximum1}
\end{align*}
for all $0<t<T$. In \cite[Proposition~3.1]{LL19GlobalUniqueness}
or \cite[Proposition~4.1]{Ros16FractionalLaplacianSurvey}, they
showed that 
\begin{equation*}
\int_{\mathbb{R}^{n}}(-\Delta)^{\frac{s}{2}}\bm{u}_{M}(t)(-\Delta)^{\frac{s}{2}}\bm{u}_{M}^{-}(t)\,\mathsf{d}x\le0\quad\text{for all }0<t<T.\label{eq:maximum2}
\end{equation*}
Combining 
the preceding two inequalities,
we then conclude
$\frac{\mathsf{d}}{\mathsf{d}t}\big(\int_{\Omega}|\bm{u}_{M}^{-}(t)|^{2}\,\mathsf{d}x\big)
\leq 0$ holds true for all $0<t<T$.
Since $u_{M}^{-}=0$ on $\mathbb{R}^{n}\times\{0\}$ (because $\varphi\ge0$
in $\mathbb{R}^{n}$), then we conclude 
$
\int_{\Omega}|\bm{u}_{M}^{-}(t)|^{2}\,\mathsf{d}x=0\quad\text{for all }0<t<T,
$
which completes our proof. 
\end{proof}
\begin{cor}
[Comparison principle] \label{cor:comparison}Given any $n\in\mathbb{N}$
and $0<s<1$. Let $\Omega\subset\mathbb{R}^{n}$ be a bounded Lipschitz
domain in $\mathbb{R}^{n}$, and let $a\in L^{\infty}(\Omega_{T})$.
Let $u_{1}$ and $u_{2}$ be weak solutions of 
\[
\begin{cases}
(\partial_{t}+(-\Delta)^{s}+a)u_{j}=F_{j} & \text{in}\;\;\Omega_{T},\\
u_{j}=f_{j} & \text{in}\;\;\Omega_{T}^{e},\\
u_{j}=\varphi_{j} & \text{on}\;\;\{0\}\times\mathbb{R}^{n},
\end{cases}
\]
for $j=1,2$. If 
$\,
F_{1}\ge F_{2}\text{ in }\Omega_{T},\quad f_{1}\ge f_{2}\text{ in }\Omega_{T}^{e},\quad\varphi_{1}\ge\varphi_{2}\text{ in }\mathbb{R}^{n},
$
then $u_{1}\ge u_{2}$ in $\Omega_{T}$. 
\end{cor}

\begin{proof}
By applying
Proposition~\ref{prop:maximum-principle} with $u=u_{1}-u_{2}$, this can be proved immediately. 
\end{proof}

\begin{rem}
Proposition~\ref{prop:maximum-principle} as well as Corollary~\ref{cor:comparison}
also imply the uniqueness part of Proposition~\ref{prop:well-posed-linear}
and Corollary~\ref{cor:well-posed-linear}. 
\end{rem}

\subsection{\texorpdfstring{$L^{\infty}$}{L\^infty}-bounds of solutions of the linear equation}

For our purposes, we require the following $L^{\infty}$-bound estimate, which can be found in \cite[Proposition~3.3]{Li22GlobalUniquenessParabolicSemilinear}:

\begin{prop} \label{prop:sup-bound}
	Given any $n\in\mathbb{N}$ and $0<s<1$. Let
	$\Omega\subset\mathbb{R}^{n}$ be a bounded Lipschitz domain in $\mathbb{R}^{n}$,
	and let $a\in L^{\infty}(\Omega_{T})$. Suppose that $u\in L^{2}(0,T;H^{s}(\mathbb{R}^{n})) \cap H^{1}(0,T;L^{2}(\Omega))$ is a weak solution of 
	\[
	\begin{cases}
	(\partial_{t}+(-\Delta)^{s}+a)u=F & \text{in}\;\;\Omega_{T},\\
	u=f & \text{in}\;\;\Omega_{T}^{e},\\
	u=0 & \text{on}\;\;\{0\}\times\mathbb{R}^{n},
	\end{cases}
	\]
	with $F\in L^{\infty}(\Omega_{T})$ and $f\in L^{\infty}(\Omega_{T}^{e})$. Then 
$
	\|u\|_{L^{\infty}(\Omega_{T})}
	\le C ( \|f\|_{L^{\infty}(\Omega_{T}^{e})}+\|F\|_{L^{\infty}(\Omega_{T})} ),
	$
	for some constant $C=C(n,s,T,\Omega,\|a\|_{L^{\infty}(\Omega_{T})})$.
\end{prop}
To make our paper more self-contained, here we sketch the proof of Proposition~\ref{prop:sup-bound}.
The following lemma can be found in \cite[Lemma~3.4]{LL19GlobalUniqueness}
(with $a\equiv0$) or \cite[Lemma~5.1]{Ros16FractionalLaplacianSurvey}. 
\begin{lem}
[Elliptic barrier]Given any $n\in\mathbb{N}$ and $0<s<1$. Let $\Omega$
be a bounded Lipschitz domain in $\mathbb{R}^{n}$. There exists a
function $\phi=\phi(x)\in\mathcal{C}_{c}^{\infty}(\mathbb{R}^{n})$
such that 
\[
(-\Delta)^{s}\phi\ge1 \quad  \text{in}\;\;\Omega, \quad
\phi\ge0 \quad \text{in}\;\;\mathbb{R}^{n},\quad
\phi\le C \quad \text{in}\;\;\Omega,\quad \mbox{for some constant $C=C(n,s,\Omega)$}.
\]
\end{lem}

If we define $\Phi(t,x):=e^{t}\phi(x)$, we immediately obtain the
following corollary: 
\begin{cor}
[Parabolic barrier] \label{cor:barrier}Given any $n\in\mathbb{N}$
and $0<s<1$. Let $\Omega$ be a bounded Lipschitz domain in $\mathbb{R}^{n}$.
There exists a function $\Phi\in\mathcal{C}_{c}^{\infty}([0,T]\times\mathbb{R}^{n})$
such that 
\[
(\partial_{t}+(-\Delta)^{s})\Phi\ge1 \quad \text{in}\;\;\Omega_{T}\quad
\Phi\ge0  \quad \text{in}\;\;[0,T)\times\mathbb{R}^{n},\quad
\Phi\le C \quad \text{in}\;\;\Omega_{T},
\]
for some constant $C=C(n,s,T,\Omega)$. 
\end{cor}

Using the barrier in Corollary~\ref{cor:barrier}, we now can obtain the following $L^{\infty}$-bound for the solution of \eqref{eq:linear-diffusion1}. 

\begin{proof}[Proof of Proposition~{\rm \ref{prop:sup-bound}}]
Using the functions given in \eqref{eq:reduction-M} with $M=\|a\|_{L^{\infty}(\Omega_{T})}$,
we know that 
\[
\begin{cases}
(\partial_{t}+(-\Delta)^{s}+a_{M})u_{M}=F_{M} & \text{in}\;\;\Omega_{T},\\
u_{M}=f_{M} & \text{in}\;\;\Omega_{T}^{e},\\
u_{M}=0 & \text{on}\;\;\{0\}\times\mathbb{R}^{n},
\end{cases}
\]
with $a_{M}\ge0$. Let 
$v(t,x):=\|f_{M}\|_{L^{\infty}(\Omega_{T}^{e})}+\|F_{M}\|_{L^{\infty}(\Omega_{T})}\Phi(t,x)\ge0\,\,\text{in}\;\;[0,T)\times\Omega,$
where $\Phi$ is the barrier given in Corollary~\ref{cor:barrier}.
We see that 
\begin{align*}
 (\partial_{t}+(-\Delta)^{s}+a_{M})v
 & \ge(\partial_{t}+(-\Delta)^{s})v
 =\|F_{M}\|_{L^{\infty}(\Omega_{T})}(\partial_{t}+(-\Delta)^{s})\Phi
 \ge \|F_{M}\|_{L^{\infty}(\Omega_{T})} \\
 & \ge \mp F_{M}= \mp (\partial_{t}+(-\Delta)^{s}+a_{M})u_{M} \quad \text{in} \quad \Omega_{T}.
\end{align*}
Moreover, we also have 
\begin{equation}
\begin{aligned}\label{eq:compare1}
&(\partial_{t}+(-\Delta)^{s}+a_{M})(v \pm u_{M})\ge0\quad\qquad\qquad\hspace{2mm}\text{in}\;\;\Omega_{T}\\
&v \pm u_{M}=v \pm f_{M}\ge\|f_{M}\|_{L^{\infty}(\Omega_{T}^{e})} \pm f_{M}\ge0\quad  \text{in}\;\;\Omega_{T}^{e}\\
&v \pm u_{M}=v\ge0\quad\quad \qquad \hspace{34mm}\text{on}\;\;\{0\}\times\mathbb{R}^{n}.
\end{aligned}
\end{equation}
Combining relations in \eqref{eq:compare1}, 
and Proposition~\ref{prop:maximum-principle}, we see that $
v\ge \pm u_{M}\quad\text{in}\;\;\Omega_{T},$
which further implies that 
$
\|u_{M}\|_{L^{\infty}(\Omega_{T})}\le\|v\|_{L^{\infty}(\Omega_{T})}\le\|f_{M}\|_{L^{\infty}(\Omega_{T}^{e})}+C\|F_{M}\|_{L^{\infty}(\Omega_{T})},$
where $C=C(n,s,T,\Omega)$ is the constant given in the  Corollary~\ref{cor:barrier}.
Finally, utilizing
\begin{align*}
|u(t,x)| & =e^{Mt}|u_{M}(t,x)|\le e^{T\|a\|_{L^{\infty}(\Omega_{T})}}\|u_{M}\|_{L^{\infty}(\Omega_{T})}\quad\text{in}\;\;\Omega_{T},\\
|F_{M}(t,x)| & =e^{-Mt}|F(t,x)|\le\|F\|_{L^{\infty}(\Omega_{T})}\quad\quad\hspace{18mm}\text{in}\;\;\Omega_{T},\\
|f_{M}(t,x)| & =e^{-Mt}|f(t,x)|\le\|f\|_{L^{\infty}(\Omega_{T}^{e})}\quad\quad\hspace{20mm}\text{in}\;\;\Omega_{T}^{e},
\end{align*}
we conclude the proof. 
\end{proof}

We skip the proof of the following well-posedness result  as it follows from combining Corollary~\ref{cor:well-posed-linear} and Proposition~\ref{prop:sup-bound}.
 
\begin{prop} \label{prop:well-posed-suitable-regularity}
Given any $n\in\mathbb{N}$
and $0<s<1$. Let $\Omega\subset\mathbb{R}^{n}$ be a bounded Lipschitz
domain in $\mathbb{R}^{n}$, let $W\subset\Omega^{e}$ be any
open set with Lipschitz boundary satisfying $\overline{W}\cap\overline{\Omega}=\emptyset$.
Then for any $\tilde{F}\in L^{\infty}(\Omega_{T})$ and $f\in\mathcal{C}_{c}^{\infty}(W_{T})$,
there exists a unique weak solution $u$ of 
\[
\begin{cases}
(\partial_{t}+(-\Delta)^{s}+a)u=F & \text{in}\;\;\Omega_{T},\\
u=f & \text{in}\;\;\Omega_{T}^{e},\\
u=0 & \text{in}\;\;\{0\}\times\mathbb{R}^{n},
\end{cases}
\]
satisfying 
\begin{align*}
 &\|u\|_{L^{\infty}(0,T;H^{s}(\mathbb{R}^{n}))\cap L^{\infty}(\mathbb{R}_{T}^{n})}^{2}+\|\partial_{t}u\|_{L^{2}(\Omega_{T})}^{2}\\
 & \quad\le  C\big(\|F\|_{L^{\infty}(\Omega_{T})}^{2}+\|f\|_{L^{\infty}(0,T;H^{s}(\mathbb{R}^{n}))\cap L^{\infty}(\mathbb{R}_{T}^{n})}^{2}+\|(-\Delta)^{s}f\|_{L^{2}(\Omega_{T})}^{2}\big)
\end{align*}
for some constant $C=C(n,s,T,\|a\|_{L^{\infty}(\Omega_{T})},\Omega)$.
\end{prop}

\subsection{Well-posedness for the  nonlinear equation}

We now state the well-posedness of \eqref{eq:nonlinear-diffusion-main}
for small exterior data: 
\begin{prop}
\label{prop:well-posedness-diffusion-nonlinear}
Given any $n\in\mathbb{N}$
and $0<s<1$. Let $\Omega\subset\mathbb{R}^{n}$ be a bounded Lipschitz
domain in $\mathbb{R}^{n}$, and $W\subset\Omega^{e}$ be any
open set with Lipschitz boundary satisfying $\overline{W}\cap\overline{\Omega}=\emptyset$.
Fixing any parameter $\delta>0$. Assume that $q$ satisfies  \ref{itm:q1}--\ref{itm:q3}.
Then there exists a sufficiently small parameter $\tilde{\epsilon}_{0}=\tilde{\epsilon}_{0}(n,s,\Omega,T,\delta)>0$
such that the following statement holds: Given any $f\in\mathcal{C}_{c}^{\infty}(W_{T})$ with $\|f\|_{{\rm ext}} \leq \tilde{\epsilon}_{0}$,
there exists a unique solution $u\in L^{\infty}(0,T;H^{s}(\mathbb{R}^{n}))\cap L^{\infty}(\mathbb{R}_{T}^{n})$
of \eqref{eq:nonlinear-diffusion-main} with 
\begin{equation}
\|u\|_{L^{\infty}(0,T;H^{s}(\mathbb{R}^{n}))\cap L^{\infty}(\mathbb{R}_{T}^{n})}\le C\|f\|_{{\rm ext}}\label{eq:continuity-dependence-solution}
\end{equation}
for certain constant $C=C(n,s,T,\Omega)$. 
\end{prop}

\begin{rem}
In order to prove Proposition~\ref{prop:well-posedness-diffusion-nonlinear},
we only need $q$ to be $\mathcal{C}^{1}$-smooth in $z$ variable. However
to recover $m$-th jet of $q$ we need to assume \ref{itm:q1}. 
\end{rem}

\begin{rem}
In \cite[Theorem~11.2]{MBRS16NonlocalFractional}, they showed that
there exist infinitely many solutions $w_{j}$ to 
\[
\begin{cases}
(-\Delta)^{s}w_{j}+q(x,w_{j})+h(x)=0 & \text{in}\;\;\Omega,\\
w_{j}=0 & \text{in}\;\;\Omega^{e},
\end{cases}
\]
such that $\|w_{j}\|_{H^{s}(\mathbb{R}^{n})}\rightarrow\infty$ as $j\rightarrow\infty$.
Therefore, the smallness assumption on $f$ seems to be necessary to ensure the uniqueness of the solution to \eqref{eq:nonlinear-diffusion-main}. 
\end{rem}

\begin{proof}
[Proof of Proposition~{\rm \ref{prop:well-posedness-diffusion-nonlinear}}]
\textbf{Step 1: Initialization.} Given any $f\in\mathcal{C}_{c}^{\infty}(\Omega_{T})$, from Proposition~\ref{prop:well-posed-suitable-regularity}, there exists
a unique solution $u_{0}=u_{0}(t,x)$ of 
\[
\begin{cases}
(\partial_{t}+(-\Delta)^{s})u=0 & \text{in}\;\;\Omega_{T},\\
u_{0}=f & \text{in}\;\;\Omega_{T}^{e},\\
u_{0}=0 & \text{in}\;\;\{0\}\times\mathbb{R}^{n},
\end{cases}
\]
with 
\begin{equation} \label{eq:diffusion-linear-shift-estimate}
\|u_{0}\|_{L^{\infty}(0,T;H^{s}(\mathbb{R}^{n}))\cap L^{\infty}(\mathbb{R}_{T}^{n})}\le C\|f\|_{{\rm ext}},
\end{equation}
for some constant $C=C(n,s,T,\Omega)$. If $u$ is a solution of \eqref{eq:nonlinear-diffusion-main},
then the remainder function $v\equiv u-u_{0}$ satisfies 
\begin{equation}
\begin{cases}
(\partial_{t}+(-\Delta)^{s})v=\mathcal{F}(v)\equiv-q(t,x,(v+u_{0})(t,x)) & \text{in}\;\;\Omega_{T},\\
v=0 & \text{in}\;\;\Omega_{T}^{e},\\
v=0 & \text{in}\;\;\{0\}\times\mathbb{R}^{n}.
\end{cases}\label{eq:diffusion-remainder}
\end{equation}
Again, using Proposition~\ref{prop:well-posed-suitable-regularity}, given
any $F=F(t,x)\in L^{\infty}(\Omega_{T})$, there exists a unique solution
$\mathcal{S}F\in L^{\infty}(0,T;\tilde{H}^{s}(\Omega))\cap L^{\infty}(\mathbb{R}_{T}^{n})$
of 
\begin{equation}
\begin{cases}
(\partial_{t}+(-\Delta)^{s})\mathcal{S}F=F & \text{in}\;\;\Omega_{T},\\
\mathcal{S}F=0 & \text{in}\;\;\Omega_{T}^{e},\\
\mathcal{S}F=0 & \text{in}\;\;\{0\}\times\mathbb{R}^{n},
\end{cases}\label{eq:diffusion-linear-iteration}
\end{equation}
with 
$\|\mathcal{S}F\|_{L^{\infty}(0,T;\tilde{H}^{s}(\Omega))\cap L^{\infty}(\mathbb{R}_{T}^{n})}\le C\|F\|_{L^{\infty}(\Omega_{T})} $
for some constant $C=C(n,s,T,\Omega)$. In other words, the solution
operator 
\begin{equation}
\mathcal{S}:L^{\infty}(\Omega_{T})\rightarrow L^{\infty}(0,T;\tilde{H}^{s}(\Omega))\cap L^{\infty}(\mathbb{R}_{T}^{n})\label{eq:solution-operator-boundedness}
\end{equation}
of \eqref{eq:diffusion-linear-iteration} is a bounded linear operator. 

\textbf{Step 2: Contraction.} Let $\epsilon=\|f\|_{{\rm ext}}$, and
we define
\[
X_{\epsilon}
:= \big\{ v\in L^{\infty}(0,T;\tilde{H}^{s}(\Omega))\cap L^{\infty}(\mathbb{R}_{T}^{n}) \,\big|\, \|v\|_{L^{\infty}(0,T;\tilde{H}^{s}(\Omega))\cap L^{\infty}(\mathbb{R}_{T}^{n})}\le\epsilon \big\}.
\]
We first show that 
\begin{equation}
\mathcal{S}\circ\mathcal{F}(v)\in X_{\epsilon}\quad\text{for all}\;\;v\in X_{\epsilon}.\label{eq:contraction1}
\end{equation}
Given any $v\in X_{\epsilon}$, using \eqref{eq:diffusion-linear-shift-estimate},
we know 
\begin{equation}
\|u_{0}+v\|_{L^{\infty}(0,T;\tilde{H}^{s}(\Omega))\cap L^{\infty}(\mathbb{R}_{T}^{n})}\le C\epsilon.\label{eq:contraction1-eq1}
\end{equation}
By choosing $\tilde{\epsilon}_{0}=\tilde{\epsilon}_{0}(n,s,\Omega,T,\delta)>0$ to be
sufficiently small, we can guarantee that $2C\epsilon\le2C\tilde{\epsilon}_{0}<\delta$.
From \ref{itm:q1} and \ref{itm:q2}, by using the mean value theorem, we can find a function
$0\le\zeta(t,x)\le1$ such that 
\begin{align}
 \mathcal{F}(v)(t,x)
 & = q(t,x,(u_{0}+v)(t,x))-q(t,x,0)\nonumber \\
 & =\partial_{z}q\big(t,x,(\zeta(u_{0}+v))(t,x)\big)(u_{0}+v)(t,x), \quad \text{for all}\;\;x\in\Omega.\label{eq:contraction1-eq2}
\end{align}
Therefore, using \ref{itm:q3}, combining \eqref{eq:contraction1-eq1} and
\eqref{eq:contraction1-eq2}, we obtain 
$
\|\mathcal{F}(v)\|_{L^{\infty}(\Omega_{T})}\le\Phi(C\epsilon)\|u_{0}+v\|_{L^{\infty}(\mathbb{R}_{T}^{n})}\le C\Phi(C\epsilon)\epsilon.
$
Using \eqref{eq:solution-operator-boundedness}, we then obtain 
$
\|\mathcal{S}\circ\mathcal{F}(v)\|_{L^{\infty}(0,T;\tilde{H}^{s}(\Omega))\cap L^{\infty}(\mathbb{R}_{T}^{n})}\le\tilde{C}\,\Phi(C\epsilon)\,\epsilon.
$
Since $\Phi$ is non-decreasing, using the assumption \ref{itm:q3}, and by choosing a smaller $\tilde{\epsilon}_{0}=\tilde{\epsilon}_{0}(n,s,\Omega,T,\delta)>0$,
we can assure that $\Phi(C\epsilon)\le\Phi(C\tilde{\epsilon}_{0})\le\tilde{C}^{-1}$,
and can obtain 
\begin{equation}
\|\mathcal{S}\circ\mathcal{F}(v)\|_{L^{\infty}(0,T;\tilde{H}^{s}(\Omega))\cap L^{\infty}(\mathbb{R}_{T}^{n})}\le\epsilon,\label{eq:contraction1-eq3}
\end{equation}
which concludes \eqref{eq:contraction1}. 

We next show that 
\begin{equation}
\mathcal{S}\circ\mathcal{F}\text{ is a contraction on }X_{\epsilon}.\label{eq:contraction2}
\end{equation}
Let $v_{1},v_{2}\in X_{\epsilon}$, similar to \eqref{eq:contraction1-eq1},
we have 
\begin{equation}
\|u_{0}+v_{j}\|_{L^{\infty}(0,T;\tilde{H}^{s}(\Omega))\cap L^{\infty}(\mathbb{R}_{T}^{n})}\le C\epsilon\quad\text{for all}\;\;j=1,2.\label{eq:contraction2-eq1}
\end{equation}
From \ref{itm:q1}, by using the mean value theorem, we can find a function $0\le\zeta(t,x)\le1$
such that 
\begin{align}
 \mathcal{F}(v_{1})(t,x)-\mathcal{F}(v_{2})(t,x)
 & = q(t,x,(u_{0}+v_{1})(t,x))-q(t,x,(u_{0}+v_{2})(t,x))\nonumber \\
 & =\partial_{z}q\big(t,x,(\zeta(u_{0}+v_{1})+(1-\zeta)(u_{0}+v_{2}))(t,x)\big)(v_{1}-v_{2})(t,x).\label{eq:contraction2-eq2}
\end{align}
Using \eqref{eq:contraction2-eq1}, we know that 
\begin{equation}
\|\zeta(u_{0}+v_{1})+(1-\zeta)(u_{0}+v_{2})\|_{L^{\infty}(\mathbb{R}_{T}^{n})}\le C\epsilon\le C\tilde{\epsilon}_{0}.\label{eq:contraction2-eq3}
\end{equation}
Since $C\tilde{\epsilon}_{0}<\delta$, using \ref{itm:q3}, combining \eqref{eq:contraction2-eq2},
we have
\begin{align*}
\|\mathcal{F}(v_{1})-\mathcal{F}(v_{2})\|_{L^{\infty}(\Omega_{T})}  & \leq \Phi(C\tilde{\epsilon}_{0}) \|v_{1}-v_{2}\|_{L^{\infty}(\Omega_{T})} \\
& \leq \Phi(C\tilde{\epsilon}_{0})\|v_{1}-v_{2}\|_{L^{\infty}(0,T;\tilde{H}^{s}(\Omega))\cap L^{\infty}(\mathbb{R}_{T}^{n})}.
\end{align*}
Using \eqref{eq:solution-operator-boundedness}, we then obtain
\[
\|\mathcal{S}\circ\mathcal{F}(v_{1})-\mathcal{S}\circ\mathcal{F}(v_{2})\|_{L^{\infty}(0,T;\tilde{H}^{s}(\Omega))\cap L^{\infty}(\mathbb{R}_{T}^{n})} \le\tilde{C}\Phi(C\tilde{\epsilon}_{0})\|v_{1}-v_{2}\|_{L^{\infty}(0,T;\tilde{H}^{s}(\Omega))\cap L^{\infty}(\mathbb{R}_{T}^{n})}.
\]
Since $\Phi$ is non-decreasing, using \ref{itm:q3}, possibly choosing a smaller $\tilde{\epsilon}_{0}=\tilde{\epsilon}_{0}(n,s,\Omega,T,\delta)>0$,
we can assure that $\Phi(C\tilde{\epsilon}_{0})\le\frac{1}{2}\tilde{C}^{-1}$,
and we obtain 
\begin{equation*}
\|\mathcal{S}\circ\mathcal{F}(v_{1})-\mathcal{S}\circ\mathcal{F}(v_{2})\|_{L^{\infty}(0,T;\tilde{H}^{s}(\Omega))\cap L^{\infty}(\mathbb{R}_{T}^{n})}
\le\frac{1}{2}\|v_{1}-v_{2}\|_{L^{\infty}(0,T;\tilde{H}^{s}(\Omega))\cap L^{\infty}(\mathbb{R}_{T}^{n})},
\end{equation*}
which concludes \eqref{eq:contraction2}. 

\textbf{Step 3: Conclusion.} From \eqref{eq:contraction1} and \eqref{eq:contraction2},
by using the Banach fixed point theorem, there exists a unique $v\in X_{\epsilon}$
such that $v=\mathcal{S}\circ\mathcal{F}(v)$, that is, there exists
a unique $v\in X_{\epsilon}$ satisfying \eqref{eq:diffusion-remainder}.
Hence we know that $u\equiv v+u_{0}\in L^{\infty}(0,T;\tilde{H}^{s}(\Omega))\cap L^{\infty}(\mathbb{R}_{T}^{n})$
is the unique solution of \eqref{eq:nonlinear-diffusion-main}. Moreover,
from \eqref{eq:diffusion-linear-shift-estimate} and \eqref{eq:contraction1-eq3},
we can conclude \eqref{eq:continuity-dependence-solution}. 
\end{proof}

\section{\label{sec:The-Runge-approximation}The Runge approximation for the
fractional diffusion equation}

The following unique continuation property for $(-\Delta)^{s}$ (see \cite{GSU20Calderon}) is crucial for our work.
\begin{lem}
[Antilocality] \label{lem:UCP}Suppose $u=(-\Delta)^{s}u=0$
in $\mathscr{O}_{T}$, for some open set $\mathscr{O}\subset\mathbb{R}^{n}$,
then $u\equiv0$ in $\mathbb{R}_{T}^{n}$. 
\end{lem}

The following Runge approximation property for the diffusion equation can be found in \cite[Proposition~2.4]{Li22GlobalUniquenessParabolicSemilinear}. To make our paper self-contained, here we still present the proof.
\begin{prop} \label{prop:Runge-diffusion}
Given any $n\in\mathbb{N}$ and $0<s<1$.
Let $\Omega\subset\mathbb{R}^{n}$ be a bounded Lipschitz domain in
$\mathbb{R}^{n}$, let $W\subset\Omega^{e}$ be any open set with
Lipschitz boundary satisfying $\overline{W}\cap\overline{\Omega}=\emptyset$.
Fixing any $a\in L^{\infty}(\Omega_{T})$. For each $f\in\mathcal{C}_{c}^{\infty}(W_{T})$,
let $\mathcal{P}_{a}f\in L^{\infty}(0,T;H^{s}(\mathbb{R}^{n}))\cap L^{\infty}(\mathbb{R}_{T}^{n})$
be the unique solution (see Proposition~{\rm \ref{prop:well-posed-suitable-regularity}})
of 
\[
\begin{cases}
(\partial_{t}+(-\Delta)^{s}+a(t,x))\mathcal{P}_{a}f=0 & \text{in}\;\;\Omega_{T},\\
\mathcal{P}_{a}f=f & \text{in}\;\;\Omega_{T}^{e},\\
\mathcal{P}_{a}f=0 & \text{on}\;\;\{0\}\times\mathbb{R}^{n}.
\end{cases}
\]
Then the set $\mathcal D := \{ \mathcal{P}_{a}f|_{\Omega_{T}} \,|\, f\in\mathcal{C}_{c}^{\infty}(W_{T}) \}$
is dense in $L^{2}(\Omega_{T})$. 
\end{prop}

\begin{proof}
Using the Hahn-Banach theorem (see e.g.~\cite[Corollary~1.8]{Bre11PDE}),
we only need to show the following: if $v\in L^{2}(\Omega_{T})$ satisfies
\begin{equation*}
(\mathcal{P}_{a}f,v)_{L^{2}(\Omega_{T})}=0\quad\text{for all}\;\;f\in\mathcal{C}_{c}^{\infty}(W_{T}),\label{eq:HB-assumption}
\end{equation*}
then $v=0$ in $\Omega_{T}$.
By Proposition \ref{prop:well-posed-linear}, there exists a unique $\tilde{w}\in L^{\infty}(0,T;H^{s}(\mathbb{R}^{n}))\cap L^{\infty}(\mathbb{R}_{T}^{n})$
such that 
\[
\begin{cases}
(\partial_{t}+(-\Delta)^{s}+a(T-t,x))\tilde{w}=v(T-t,x) & \text{in}\;\;\Omega_{T},\\
\tilde{w}=0 & \text{in}\;\;\Omega_{T}^{e},\\
\tilde{w}=0 & \text{on}\;\;\{0\}\times\mathbb{R}^{n}.
\end{cases}
\]
Define $w(t,x):=\tilde{w}(T-t,x)$, then 
\begin{equation}
\begin{cases}
(-\partial_{t}+(-\Delta)^{s}+a(t,x))w=v(t,x) & \text{in}\;\;\Omega_{T},\\
w=0 & \text{in}\;\;\Omega_{T}^{e},\\
w=0 & \text{on}\;\;\{T\}\times\mathbb{R}^{n}.
\end{cases}\label{eq:Runge-aux}
\end{equation}
We note that 
\begin{align*}
(\mathcal{P}_{a}f,v)_{L^{2}(\Omega_{T})}
 & =(\mathcal{P}_{a}f-f,v)_{L^{2}(\Omega_{T})} \qquad \text{(because }{\rm supp}\,(f)\cap\overline{\Omega_{T}}=\emptyset)\\
 & =(\mathcal{P}_{a}f-f,(-\partial_{t}+(-\Delta)^{s}+a(t,x))w)_{L^{2}(\Omega_{T})}\\
 & =-(\mathcal{P}_{a}f,\partial_{t}w)_{L^{2}(\Omega_{T})}+(\mathcal{P}_{a}f,aw)_{L^{2}(\Omega_{T})} \qquad \text{(because }{\rm supp}\,(f)\cap\overline{\Omega_{T}}=\emptyset)\\
 & \quad+(\mathcal{P}_{a}f-f,(-\Delta)^{s}w)_{L^{2}(\mathbb{R}^{n})} \qquad \text{(because }{\rm supp}\,(\mathcal{P}_{a}f-f)\subset\overline{\Omega_{T}})\\
 & =((\partial_{t}+(-\Delta)^{s}+a)\mathcal{P}_{a}f,w)_{L^{2}(\Omega_{T})}-(f,(-\Delta)^{s}w)_{L^{2}(\mathbb{R}^{n})}\\
 & =-(f,(-\Delta)^{s}w)_{L^{2}(W_{T})}.
\end{align*}
Combining this equality with $(\mathcal{P}_{a}f,v)_{L^{2}(\Omega_{T})}=0 $, we obtain
$
(f,(-\Delta)^{s}w)_{L^{2}(W_{T})}=0$ for all $f\in\mathcal{C}_{c}^{\infty}(W_{T}),
$
which implies $(-\Delta)^{s}w=0$ in $W_{T}$. Since $w=0$ in $W_{T}$,
using Lemma~\ref{lem:UCP}, we conclude $w\equiv0$ in $\mathbb{R}_{T}^{n}$,
and hence from \eqref{eq:Runge-aux}, we conclude that $v=0$ in $\Omega_{T}$. 
\end{proof}

\section{The inverse problems for the
fractional diffusion equation} \label{sec:inverse-problem-diffusion}

In this section we perform higher order linearizations to
the nonlinear fractional diffusion equation \eqref{eq:nonlinear-diffusion-main}
as well as the DN map \eqref{eq:DN-map}, which is also nonlinear.
For each linearization step we derive certain identities and combine them with the Runge approximation to recover partial derivatives of $q$.
We start with the zeroth order linearization. 

\subsection{Zeroth order linearization}

Let $u_{j}^{\bm{\epsilon}}$ be the unique solution of 
\begin{equation}
\begin{cases}
\partial_{t}u_{j}^{\bm{\epsilon}}+(-\Delta)^{s}u_{j}^{\bm{\epsilon}}+q_{j}(\cdot,u_{j}^{\bm{\epsilon}})=0 & \text{in}\;\;\Omega_{T},\\
u_{j}^{\bm{\epsilon}}=\bm{\epsilon}\cdot\bm{g}=\epsilon_{1}g_{1}+\cdots+\epsilon_{m}g_{m} & \text{in}\;\;\Omega_{T}^{e},\\
u_{j}^{\bm{\epsilon}}=0 & \text{on}\;\;\{0\}\times\mathbb{R},
\end{cases}\label{eq:uj-epsilon1}
\end{equation}
where $\bm{g}=(g_{1},\cdots,g_{m})\in(\mathcal{C}_{c}^{\infty}(W_{T}))^{m}$.
 Since $q_{j}$ ($j=1,2$) satisfies \ref{itm:q1}--\ref{itm:q3},
there exists a constant $\epsilon_{0}=\epsilon_{0}(n,s,\Omega,T,\delta,\bm{g})>0$
with $\epsilon_{0}\le\tilde{\epsilon}_{0}$, where $\tilde{\epsilon}_{0}$
is the constant given in Proposition~\ref{prop:well-posedness-diffusion-nonlinear},
such that the following statement holds:
Given any $\bm{\epsilon}$
with 
$
|\bm{\epsilon}|=\max_{1\le k\le m}|\epsilon_{k}|<\epsilon_{0},
$
there exists a unique solution $u_{j}^{\bm{\epsilon}}\in L^{\infty}(0,T;\tilde{H}^{s}(\Omega))\cap L^{\infty}(\mathbb{R}_{T}^{n})$
of \eqref{eq:uj-epsilon1} with 
\begin{equation}
\|u_{j}^{\bm{\epsilon}}\|_{L^{\infty}(0,T;H^{s}(\mathbb{R}^{n}))\cap L^{\infty}(\mathbb{R}_{T}^{n})}\le C(n,s,\Omega,T,\bm{g},m)|\bm{\epsilon}|.\label{eq:uj-epsilon1-continuous-dependence}
\end{equation}
Therefore, the corresponding DN-map is $
\Lambda_{q_{j}}(\bm{\epsilon}\cdot\bm{g})=(-\Delta)^{s}u_{j}^{\bm{\epsilon}}\big|_{V_{T}}$ for all $0 \leq |\bm{\epsilon}|<\epsilon_{0}$.
We now show that $\bm{\epsilon}\rightarrow u_{j}^{\bm{\epsilon}}$
is continuous in the following sense: 
\begin{lem}
\label{lem:uj-epsilon1-continuity}The mapping $\bm{\epsilon}\rightarrow u_{j}^{\bm{\epsilon}}$
is continuous in $L^{\infty}(0,T;\tilde{H}^{s}(\Omega))$, that is,
\begin{equation}
\lim_{|\bm{\theta}|\rightarrow0}\|u_{j}^{\bm{\epsilon}+\bm{\theta}}-u_{j}^{\bm{\epsilon}}\|_{L^{\infty}(0,T;H^{s}(\mathbb{R}^{n}))\cap L^{\infty}(\mathbb{R}_{T}^{n})}=0\quad\text{for each }|\bm{\epsilon}|<\epsilon_{0}.\label{eq:uj-epsilon1-continuity}
\end{equation}
\end{lem}

\begin{proof}
Let $\bm{\theta}=(\theta_{1},\cdots,\theta_{m})\in\mathbb{R}^{m}$
with $|\bm{\theta}|\le|\bm{\epsilon}|$ and $|\bm{\epsilon}|+|\bm{\theta}|<\epsilon_{0}$.
We define $\delta_{\bm{\theta}}u_{j}^{\bm{\epsilon}}=u_{j}^{\bm{\epsilon}+\bm{\theta}}-u_{j}^{\bm{\epsilon}}$,
and observe that 
\begin{equation*}
\begin{cases}
(\partial_{t}+(-\Delta)^{s})\delta_{\bm{\theta}}u_{j}^{\bm{\epsilon}}=\mathcal{G} & \text{in}\;\;\Omega_{T},\\
\delta_{\bm{\theta}}u_{j}^{\bm{\epsilon}}=\bm{\theta}\cdot\bm{g} & \text{in}\;\;\Omega_{T}^{e},\\
\delta_{\bm{\theta}}u_{j}^{\bm{\epsilon}}=0 & \text{on}\;\;\{0\}\times\mathbb{R}^{n},
\end{cases} 
\end{equation*}
where 
$\mathcal{G}=-q_{j}(\cdot,u_{j}^{\bm{\epsilon}+\bm{\theta}})+q_{j}(\cdot,u_{j}^{\bm{\epsilon}}).$
From Proposition~\ref{prop:well-posed-suitable-regularity}, we know that
\begin{equation}
\|\delta_{\bm{\theta}}u_{j}^{\bm{\epsilon}}\|_{L^{\infty}(0,T;H^{s}(\mathbb{R}^{n}))\cap L^{\infty}(\mathbb{R}_{T}^{n})}\le C\big(\|\mathcal{G}\|_{L^{\infty}(\Omega_{T})}+|\bm{\theta}|\big).\label{eq:uj-epsilon1-continuity-eq2}
\end{equation}
Using mean value theorem on the $z$ variable of $q$, there exists
$0\le\zeta(t,x)\le1$ such that 
$
\mathcal{G}=-\partial_{z}q_{j}(\cdot,\zeta u_{j}^{\bm{\epsilon}+\bm{\theta}}+(1-\zeta)u_{j}^{\bm{\epsilon}})\delta_{\bm{\theta}}u_{j}^{\bm{\epsilon}}\,\,\text{in}\;\;\Omega_{T}.
$
From \eqref{eq:uj-epsilon1-continuous-dependence}, we know that 
\[
\|\zeta u_{j}^{\bm{\epsilon}+\bm{\theta}}+(1-\zeta)u_{j}^{\bm{\epsilon}}\|_{L^{\infty}(0,T;H^{s}(\mathbb{R}^{n}))\cap L^{\infty}(\mathbb{R}_{T}^{n})}\le C|\bm{\epsilon}|\quad\text{(because }|\bm{\theta}|\le|\bm{\epsilon}|).
\]
Using \ref{itm:q3}, we see that 
\[
\|\mathcal{G}\|_{L^{\infty}(\Omega_{T})}
\le \Phi(C|\bm{\epsilon}|)\|\delta_{\bm{\theta}}u_{j}^{\bm{\epsilon}}\|_{L^{\infty}(\Omega_{T})}
\le \Phi (C|\bm{\epsilon}|) \|\delta_{\bm{\theta}}u_{j}^{\bm{\epsilon}} \|_{L^{\infty}(0,T;\tilde{H}^{s}(\Omega))\cap L^{\infty}(\mathbb{R}_{T}^{n})}.
\]
Substituting this inequality into \eqref{eq:uj-epsilon1-continuity-eq2},
we obtain 
\[
\|\delta_{\bm{\theta}}u_{j}^{\bm{\epsilon}}\|_{L^{\infty}(0,T;H^{s}(\mathbb{R}^{n}))\cap L^{\infty}(\mathbb{R}_{T}^{n})}
\leq \tilde{C}\Phi(C|\bm{\epsilon}|)\|\delta_{\bm{\theta}}u_{j}^{\bm{\epsilon}}\|_{L^{\infty}(0,T;H^{s}(\mathbb{R}^{n}))\cap L^{\infty}(\mathbb{R}_{T}^{n})}+C|\bm{\theta}|.
\]
Since $\Phi$ is non-decreasing, using the assumption \ref{itm:q3}, and possibly by choosing a smaller constant $\epsilon_{0}=\epsilon_{0}(n,s,\Omega,T,\delta,\bm{g},m)>0$,
we can assure $\Phi(C|\bm{\epsilon}|)\le\Phi(C\epsilon_{0})\le\frac{1}{2}\tilde{C}^{-1}$,
thus
$
\|\delta_{\bm{\theta}}u_{j}^{\bm{\epsilon}}\|_{L^{\infty}(0,T;H^{s}(\mathbb{R}^{n}))\cap L^{\infty}(\mathbb{R}_{T}^{n})}\le2C|\bm{\theta}|,$
which implies \eqref{eq:uj-epsilon1-continuity}.
This completes the  proof.
\end{proof}

By setting $\bm{\epsilon}=0$ in \eqref{eq:uj-epsilon1}, we obtain 
\begin{align*}
&\partial_{t}u_{j}^{0}+(-\Delta)^{s}u_{j}^{0}+q_{j}(\cdot,u_{j}^{0})=0 \quad \text{in}\;\;\Omega_{T},\\
&u_{j}^{0}=0 \quad \text{in}\;\;\Omega_{T}^{e},\quad
u_{j}^{0}=0 \quad \text{on}\;\;\{0\}\times\mathbb{R}.
\end{align*}
From \eqref{eq:uj-epsilon1-continuous-dependence}, we know that $
u_{j}^{0} \equiv 0$ in $\mathbb{R}_{T}^{n}$.

\subsection{\label{subsec:First-order-linearization}First order linearization}

Assuming the derivative $\partial_{\epsilon_{1}}$ to
\eqref{eq:uj-epsilon1} is well-defined,
we obtain 
\begin{equation}
\begin{cases}
(\partial_{t}+(-\Delta)^{s}+\partial_{z}q_{j}(\cdot,u_{j}^{\bm{\epsilon}}))(\partial_{\epsilon_{1}}u_{j}^{\bm{\epsilon}})=0 & \text{in}\;\;\Omega_{T},\\
\partial_{\epsilon_{1}}u_{j}^{\bm{\epsilon}}=g_{1} & \text{in}\;\;\Omega_{T}^{e},\\
\partial_{\epsilon_{1}}u_{j}^{\bm{\epsilon}}=0 & \text{on}\;\;\{0\}\times\mathbb{R}^{n}.
\end{cases}\label{eq:formal-first-linearization}
\end{equation}
Using \ref{itm:q3}, we know that $
\|\partial_{z}q_{j}(\cdot,u_{j}^{\bm{\epsilon}})\|_{L^{\infty}(\Omega_{T})}\le\Phi(C\epsilon_{0})\le1.$
Therefore, using Proposition~\ref{prop:well-posed-suitable-regularity},
given any $\bm{\epsilon}$ with $|\bm{\epsilon}|<\epsilon_{0}$, there
exists a unique solution $v_{j}^{\bm{\epsilon}}\in L^{\infty}(0,T;\tilde{H}^{s}(\Omega))\cap L^{\infty}(\mathbb{R}_{T}^{n})$
to \eqref{eq:formal-first-linearization} with 
\begin{equation}
\|v_{j}^{\bm{\epsilon}}\|_{L^{\infty}(0,T;H^{s}(\mathbb{R}^{n}))\cap L^{\infty}(\mathbb{R}_{T}^{n})}\le C\|g_{1}\|_{{\rm ext}}.\label{eq:first-derivative-aux}
\end{equation}
Here, $v_{j}^{\bm{\epsilon}}$ is just a intermediate function which will be dropped after showing that $\partial_{\epsilon_{1}}u_{j}^{\bm{\epsilon}}$
is well-defined. 
\begin{lem}
\label{lem:first-derivative-well-defined}There exists a constant
$\epsilon_{0}=\epsilon_{0}(n,s,\Omega,T,\delta,\bm{g},m)>0$ with
$0<\epsilon_{0}<\tilde{\epsilon}_{0}$, where $\tilde{\epsilon}_{0}$
is given in Proposition~{\rm \ref{prop:well-posedness-diffusion-nonlinear}},
such that for each $\bm{\epsilon}$ with $|\bm{\epsilon}|<\epsilon_{0}$,
we have 
\begin{equation*}
\lim_{\epsilon_{1}\rightarrow0}\|v_{j}^{\bm{\epsilon}}-\delta_{\epsilon_{1}}u_{j}^{\bm{\epsilon}}\|_{L^{\infty}(0,T;H^{s}(\mathbb{R}^{n}))\cap L^{\infty}(\mathbb{R}_{T}^{n})}=0,\label{eq:first-derivative-well-defined}
\end{equation*}
where 
$
\delta_{\epsilon_{1}}u_{j}^{\bm{\epsilon}}=\frac{u_{j}^{\bm{\epsilon}+\epsilon_{1}\bm{e}_{1}}-u_{j}^{\bm{\epsilon}}}{\epsilon_{1}}\quad\text{for all}\;\;(t,x)\in\Omega_{T},
$
provided that $|\bm{\epsilon}|+|\epsilon_{1}|<\epsilon_{0}$. 
\end{lem}

\begin{proof}
Let $\epsilon_{1}$ satisfies $|\epsilon_{1}|\le|\bm{\epsilon}|$
and $|\bm{\epsilon}|+|\epsilon_{1}|<\epsilon_{0}$. Note that 
\begin{equation*}
\begin{cases}
(\partial_{t}+(-\Delta)^{s})(v_{j}^{\bm{\epsilon}}-\delta_{\epsilon_{1}}u_{j}^{\bm{\epsilon}})=\mathcal{G}_{1} & \text{in}\;\;\Omega_{T},\\
(v_{j}^{\bm{\epsilon}}-\delta_{\epsilon_{1}}u_{j}^{\bm{\epsilon}})=0 & \text{in}\;\;\Omega_{T}^{e} \;\; \text{and on} \;\; \{0\} \times \mathbb R^n,
\end{cases} 
\end{equation*}
with 
$
-\mathcal{G}_{1}=\partial_{z}q(\cdot,u_{j}^{\bm{\epsilon}})v_{j}^{\bm{\epsilon}}-\frac{q_{j}(\cdot,u_{j}^{\bm{\epsilon}+\epsilon_{1}\bm{e}_{1}})-q_{j}(\cdot,u_{j}^{\bm{\epsilon}})}{\epsilon_{1}}.
$
From Proposition~\ref{prop:well-posed-suitable-regularity}, we know that
\begin{equation}
\|v_{j}^{\bm{\epsilon}}-\delta_{\epsilon_{1}}u_{j}^{\bm{\epsilon}}\|_{L^{\infty}(0,T;\tilde{H}^{s}(\Omega))\cap L^{\infty}(\mathbb{R}_{T}^{n})}\le C\|\mathcal{G}_{1}\|_{L^{\infty}(\Omega_{T})}.\label{eq:first-derivative-eq2}
\end{equation}
Using the mean value theorem on the $z$ variable of $q$, there exists
$0\le\zeta(t,x)\le1$ such that 
\begin{align*}
-\mathcal{G}_{1}
 & =\big[\partial_{z}q(\cdot,u_{j}^{\bm{\epsilon}})-\partial_{z}q_{j}(\cdot,\zeta u_{j}^{\bm{\epsilon}+\epsilon_{1}\bm{e}_{1}}+(1-\zeta)u_{j}^{\bm{\epsilon}})\big]v_{j}^{\bm{\epsilon}}\\
 & \quad+\partial_{z}q_{j}(\cdot,\zeta u_{j}^{\bm{\epsilon}+\epsilon_{1}\bm{e}_{1}}+(1-\zeta)u_{j}^{\bm{\epsilon}})\big[v_{j}^{\bm{\epsilon}}-\delta_{\epsilon_{1}}u_{j}^{\bm{\epsilon}}\big].
\end{align*}
Using mean value theorem on the $z$ variable of $\partial_{z}q$,
there exists $0\le\eta(t,x)\le1$ such that 
\begin{align*}
-\mathcal{G}_{1}
 & =-\zeta\partial_{z}^{2}q\big(\eta u_{j}^{\bm{\epsilon}}-(1-\eta)(\zeta u_{j}^{\bm{\epsilon}+\epsilon_{1}\bm{e}_{1}}+(1-\zeta)u_{j}^{\bm{\epsilon}})\big)(u_{j}^{\bm{\epsilon}+\epsilon_{1}\bm{e}_{1}}-u_{j}^{\bm{\epsilon}})v_{j}^{\bm{\epsilon}}\\
 & \quad+\partial_{z}q_{j}(\cdot,\zeta u_{j}^{\bm{\epsilon}+\epsilon_{1}\bm{e}_{1}}+(1-\zeta)u_{j}^{\bm{\epsilon}})\big[v_{j}^{\bm{\epsilon}}-\delta_{\epsilon_{1}}u_{j}^{\bm{\epsilon}}\big]
\end{align*}
From \eqref{eq:uj-epsilon1-continuous-dependence} and $|\epsilon_{1}|\le|\bm{\epsilon}|$,
we have 
\begin{align*}
\|\eta u_{j}^{\bm{\epsilon}}-(1-\eta)(\zeta u_{j}^{\bm{\epsilon}+\epsilon_{1}\bm{e}_{1}}+(1-\zeta)u_{j}^{\bm{\epsilon}})\|_{L^{\infty}(0,T;H^{s}(\mathbb{R}^{n}))\cap L^{\infty}(\mathbb{R}_{T}^{n})} & \le C|\bm{\epsilon}|,\\
\|\zeta u_{j}^{\bm{\epsilon}+\epsilon_{1}\bm{e}_{1}}+(1-\zeta)u_{j}^{\bm{\epsilon}}\|_{L^{\infty}(0,T;H^{s}(\mathbb{R}^{n}))\cap L^{\infty}(\mathbb{R}_{T}^{n})} & \le C|\bm{\epsilon}|.
\end{align*}
Hence, by \ref{itm:q3} and \ref{itm:q4}, we know that 
\begin{align*}
\big\|\zeta\partial_{z}^{2}q\big(\eta u_{j}^{\bm{\epsilon}}-(1-\eta)(\zeta u_{j}^{\bm{\epsilon}+\epsilon_{1}\bm{e}_{1}}+(1-\zeta)u_{j}^{\bm{\epsilon}})\big)\big\|_{L^{\infty}(\Omega_{T})} & \le M_{2},\\
\|\partial_{z}q_{j}(\cdot,\zeta u_{j}^{\bm{\epsilon}+\epsilon_{1}\bm{e}_{1}}+(1-\zeta)u_{j}^{\bm{\epsilon}})\|_{L^{\infty}(\Omega_{T})} & \le\Phi(C|\bm{\epsilon}|).
\end{align*}
Hence, by using \eqref{eq:first-derivative-aux} we know that 
\begin{align*}
\|\mathcal{G}_{1}\|_{L^{\infty}(\Omega_{T})}
& \le CM_{2}\|g_{1}\|_{{\rm ext}}\|u_{j}^{\bm{\epsilon}+\epsilon_{1}\bm{e}_{1}}-u_{j}^{\bm{\epsilon}}\|_{L^{\infty}(\Omega_{T})} + \Phi(C|\bm{\epsilon}|)\|v_{j}^{\bm{\epsilon}}-\delta_{\epsilon_{1}}u_{j}^{\bm{\epsilon}}\|_{L^{\infty}(\Omega_{T})}.
\end{align*}
Combining this with \eqref{eq:first-derivative-eq2}, we
have 
\begin{align*}
 & \|v_{j}^{\bm{\epsilon}}-\delta_{\epsilon_{1}}u_{j}^{\bm{\epsilon}}\|_{L^{\infty}(0,T;\tilde{H}^{s}(\Omega))\cap L^{\infty}(\mathbb{R}_{T}^{n})}\\
 & \le \tilde{C}M_{2}\|g_{1}\|_{{\rm ext}}\|u_{j}^{\bm{\epsilon}+\epsilon_{1}\bm{e}_{1}}-u_{j}^{\bm{\epsilon}}\|_{L^{\infty}(\Omega_{T})}+ \Phi(C|\bm{\epsilon}|)\|v_{j}^{\bm{\epsilon}}-\delta_{\epsilon_{1}}u_{j}^{\bm{\epsilon}}\|_{L^{\infty}(0,T;\tilde{H}^{s}(\Omega))\cap L^{\infty}(\mathbb{R}_{T}^{n})}.
\end{align*}
Since $\Phi$ is non-decreasing, using the limiting assumption of
$\Phi$ in \ref{itm:q3}, possibly choosing a smaller $\epsilon_{0}=\epsilon_{0}(n,s,\Omega,T,\delta,\bm{g},m)>0$,
we can assure that $\Phi(C|\bm{\epsilon}|)\le\frac{1}{2}$, and hence
\[
\|v_{j}^{\bm{\epsilon}}-\delta_{\epsilon_{1}}u_{j}^{\bm{\epsilon}}\|_{L^{\infty}(0,T;\tilde{H}^{s}(\Omega))\cap L^{\infty}(\mathbb{R}_{T}^{n})}\le\tilde{C}M_{2}\|g_{1}\|_{{\rm ext}}\|u_{j}^{\bm{\epsilon}+\epsilon_{1}\bm{e}_{1}}-u_{j}^{\bm{\epsilon}}\|_{L^{\infty}(\Omega_{T})}.
\]
Finally, using Lemma~\ref{lem:uj-epsilon1-continuity}, we conclude Lemma
\ref{eq:first-derivative-well-defined}. 
\end{proof}


Using \ref{itm:q4}, we also see that $\partial_{\epsilon_{1}}u_{j}^{\bm{\epsilon}}|_{\bm{\epsilon}=0}$
satisfies 
\begin{equation}
\begin{cases}
(\partial_{t}+(-\Delta)^{s})(\partial_{\epsilon_{1}}u_{j}^{\bm{\epsilon}}|_{\bm{\epsilon}=0})=0 & \text{in}\;\;\Omega_{T},\\
\partial_{\epsilon_{1}}u_{j}^{\bm{\epsilon}}|_{\bm{\epsilon}=0}=g_{1} & \text{in}\;\;\Omega_{T}^{e},\\
\partial_{\epsilon_{1}}u_{j}^{\bm{\epsilon}}|_{\bm{\epsilon}=0}=0 & \text{on}\;\;\{0\}\times\mathbb{R}^{n}.
\end{cases}\label{eq:first-derivative-epsilon-zero-equation}
\end{equation}
By uniqueness of solutions (see Proposition~\ref{prop:well-posed-suitable-regularity}),
we know that $
\partial_{\epsilon_{1}}u_{1}^{\bm{\epsilon}}|_{\bm{\epsilon}=0}=\partial_{\epsilon_{1}}u_{2}^{\bm{\epsilon}}|_{\bm{\epsilon}=0}\text{in}\;\;\Omega_{T}.$
For later convenience, we simply denote 
\begin{equation}
\partial_{\epsilon_{1}}u^{\bm{\epsilon}}|_{\bm{\epsilon}=0}=\partial_{\epsilon_{1}}u_{1}^{\bm{\epsilon}}|_{\bm{\epsilon}=0}=\partial_{\epsilon_{1}}u_{2}^{\bm{\epsilon}}|_{\bm{\epsilon}=0}\text{in}\;\;\Omega_{T}.\label{eq:first-derivative-epsilon-zero}
\end{equation}

In the next lemma we show that the information from DN-map can be passed
to the first-order linearized DN-map: 
\begin{lem}
\label{lem:information-first-linearized-DN-map}If $\Lambda_{q_{1}}(f)=\Lambda_{q_{1}}(f)$
for all $f\in\mathcal{C}_{c}^{\infty}(W_{T})$ with $\|f\|_{{\rm ext}}\le\tilde{\epsilon}_{0}$,
where $\tilde{\epsilon}_{0}$ is the constant given in Proposition~{\rm \ref{prop:well-posedness-diffusion-nonlinear}},
then there exists a constant $\epsilon_{0}=\epsilon_{0}(n,s,\Omega,T,\delta,\bm{g},m)>0$
with $0<\epsilon_{0}<\tilde{\epsilon}_{0}$ such that 
\begin{equation}
(-\Delta)^{s}\partial_{\epsilon_{1}}u_{1}^{\bm{\epsilon}}\big|_{V_{T}}=(-\Delta)^{s}\partial_{\epsilon_{1}}u_{2}^{\bm{\epsilon}}\big|_{V_{T}}\label{eq:information-first-linearized-DN-map} \quad \mbox{for all} \quad \epsilon \quad \mbox{with} \quad  |\bm{\epsilon}|\le\epsilon_{0}.
\end{equation}
\end{lem}

\begin{proof}
We have 
\begin{align*}
& \big\|(-\Delta)^{s}\partial_{\epsilon_{1}}u_{j}^{\bm{\epsilon}}-\frac{\Lambda_{q_{j}}((\bm{\epsilon}+\epsilon_{1}\bm{e}_{1})\cdot\bm{g})-\Lambda_{q_{j}}(\bm{\epsilon}\cdot\bm{g})}{\epsilon_{1}}\big\|_{L^{\infty}(0,T;H^{-s}(V))}\\
= & \quad \big\|(-\Delta)^{s}\big(\partial_{\epsilon_{1}}u_{j}^{\bm{\epsilon}}-\frac{u_{j}^{\bm{\epsilon}+\epsilon_{1}\bm{e}_{1}}-u_{j}^{\bm{\epsilon}}}{\epsilon_{1}}\big)\big\|_{L^{\infty}(0,T;H^{-s}(V))}\\
\leq & \quad \big\|(-\Delta)^{s}\big(\partial_{\epsilon_{1}}u_{j}^{\bm{\epsilon}}-\frac{u_{j}^{\bm{\epsilon}+\epsilon_{1}\bm{e}_{1}}-u_{j}^{\bm{\epsilon}}}{\epsilon_{1}}\big)\big\|_{L^{\infty}(0,T;H^{-s}(\mathbb{R}^{n}))} \\
\leq & \quad C\big\|\big(\partial_{\epsilon_{1}}u_{j}^{\bm{\epsilon}}-\frac{u_{j}^{\bm{\epsilon}+\epsilon_{1}\bm{e}_{1}}-u_{j}^{\bm{\epsilon}}}{\epsilon_{1}}\big)\big\|_{L^{\infty}(0,T;H^{s}(\mathbb{R}^{n}))}.
\end{align*}
From Lemma~\ref{lem:first-derivative-well-defined}, we have 
\[
\lim_{\epsilon_{1}\rightarrow0}\big\|(-\Delta)^{s}\partial_{\epsilon_{1}}u_{j}^{\bm{\epsilon}}-\frac{\Lambda_{q_{j}}((\bm{\epsilon}+\epsilon_{1}\bm{e}_{1})\cdot\bm{g})-\Lambda_{q_{j}}(\bm{\epsilon}\cdot\bm{g})}{\epsilon_{1}}\big\|_{L^{\infty}(0,T;H^{-s}(V))}=0.
\]
Combining this equality with the assumption $\Lambda_{q_{1}}=\Lambda_{q_{2}}$,
we conclude \eqref{eq:information-first-linearized-DN-map}. 
\end{proof}

\subsection{Second order linearization}

First of all, we recall \eqref{eq:first-derivative-aux}: 
\begin{equation}
\|\partial_{\epsilon_{p}}u_{j}^{\bm{\epsilon}}\|_{L^{\infty}(0,T;H^{s}(\mathbb{R}^{n}))\cap L^{\infty}(\mathbb{R}_{T}^{n})}\le C\|g_{p}\|_{{\rm ext}}\quad \mbox{for $p=1,2$}.
\label{eq:first-derivative-bound}
\end{equation}
see Lemma~\ref{lem:first-derivative-well-defined}. 
Acting a formal differential operator $\partial_{\epsilon_{2}}$ on
\eqref{eq:formal-first-linearization}, we obtain 
\begin{equation} \label{eq:formal-second-linearization}
\begin{cases}
\begin{array}{l}
(\partial_{t}+(-\Delta)^{s}+\partial_{z}q_{j}(\cdot,u_{j}^{\bm{\epsilon}}))(\partial_{\epsilon_{1}\epsilon_{2}}u_{j}^{\bm{\epsilon}})\\
\quad+\partial_{z}^{2}q_{j}(\cdot,u_{j}^{\bm{\epsilon}})(\partial_{\epsilon_{1}}u_{j}^{\bm{\epsilon}})(\partial_{\epsilon_{2}}u_{j}^{\bm{\epsilon}})=0
\end{array} & \mbox{in}\;\;\Omega_{T},\\
\partial_{\epsilon_{1}\epsilon_{2}}u_{j}^{\bm{\epsilon}}=0 & \text{in}\;\;\Omega_{T}^{e} \;\; \text{and on} \;\; \{0\} \times \mathbb R^n.
\end{cases}
\end{equation}
Since the term $\partial_{z}^{2}q_{j}(\cdot,u_{j}^{\bm{\epsilon}})(\partial_{\epsilon_{1}}u_{j}^{\bm{\epsilon}})(\partial_{\epsilon_{2}}u_{j}^{\bm{\epsilon}})$
is bounded in $\Omega_{T}$, using Proposition~\ref{prop:well-posed-suitable-regularity},
there exists a unique solution $
v_{j}^{\bm{\epsilon}}\in L^{\infty}(0,T;\tilde{H}^{s}(\Omega))\cap L^{\infty}(\mathbb{R}_{T}^{n})$
to \eqref{eq:formal-second-linearization} with 
\begin{align}
 \|v_{j}^{\bm{\epsilon}}\|_{L^{\infty}(0,T;\tilde{H}^{s}(\Omega))\cap L^{\infty}(\mathbb{R}_{T}^{n})}
 & \le C\|\partial_{z}^{2}q_{j}(\cdot,u_{j}^{\bm{\epsilon}})(\partial_{\epsilon_{1}}u_{j}^{\bm{\epsilon}})(\partial_{\epsilon_{2}}u_{j}^{\bm{\epsilon}})\|_{L^{\infty}(\Omega_{T})}\nonumber \\
 & \le CM_{2}\|g_{1}\|_{{\rm ext}}\|g_{2}\|_{{\rm ext}}. \qquad \text{(using \ref{itm:q4} and \eqref{eq:first-derivative-bound})} \label{eq:second-derivative-aux}
\end{align}
Again, $v_{j}^{\bm{\epsilon}}$ is temporary notation, which
will be dropped after showing $\partial_{\epsilon_{1}\epsilon_{2}}u_{j}^{\bm{\epsilon}}$
is well-defined. We emphasize that we have already dropped
$v_{j}^{\bm{\epsilon}}$, so this will not conflict with the one used in Section~\ref{subsec:First-order-linearization}. 
\begin{lem}
\label{lem:second-derivative-well-defined}There exists a constant
$\epsilon_{0}=\epsilon_{0}(n,s,\Omega,T,\delta,\bm{g},m)>0$ with
$0<\epsilon_{0}<\tilde{\epsilon}_{0}$, where $\tilde{\epsilon}_{0}$
is given in Proposition~{\rm \ref{prop:well-posedness-diffusion-nonlinear}},
such that for each $\bm{\epsilon}$ with $|\bm{\epsilon}|<\epsilon_{0}$,
we have 
\begin{equation}
\lim_{\epsilon_{2}\rightarrow0}\|v_{j}^{\bm{\epsilon}}-\delta_{\epsilon_{2}}\partial_{\epsilon_{1}}u_{j}^{\bm{\epsilon}}\|_{L^{\infty}(0,T;H^{s}(\mathbb{R}^{n}))\cap L^{\infty}(\mathbb{R}_{T}^{n})}=0,\label{eq:second-derivative-well-defined}
\end{equation}
where $
\delta_{\epsilon_{2}}\partial_{\epsilon_{1}}u_{j}^{\bm{\epsilon}}=\frac{\partial_{\epsilon_{1}}u_{j}^{\bm{\epsilon}+\epsilon_{2}\bm{e}_{2}}-\partial_{\epsilon_{1}}u_{j}^{\bm{\epsilon}}}{\epsilon_{2}}\quad\text{in}\;\;\Omega_{T},$
provided $|\bm{\epsilon}|+|\epsilon_{2}|<\epsilon_{0}$. 
\end{lem}

\begin{proof}
Let $\epsilon_{2}$ satisfies $|\epsilon_{2}|\le|\bm{\epsilon}|$
and $|\bm{\epsilon}|+|\epsilon_{2}|<\epsilon_{0}$. Note that 
\[
\begin{cases}
(\partial_{t}+(-\Delta)^{s})(v_{j}^{\bm{\epsilon}}-\delta_{\epsilon_{2}}\partial_{\epsilon_{1}}u_{j}^{\bm{\epsilon}})=\mathcal{G}_{2} & \text{in}\;\;\Omega_{T},\\
v_{j}^{\bm{\epsilon}}-\delta_{\epsilon_{2}}\partial_{\epsilon_{1}}u_{j}^{\bm{\epsilon}}=0 & \text{in}\;\;\Omega_{T}^{e} \;\; \text{and on} \;\; \{0\} \times \mathbb R^n,
\end{cases}
\]
where 
\begin{align*}
-\mathcal{G}_{2} & =\partial_{z}q_{j}(\cdot,u_{j}^{\bm{\epsilon}})v_{j}^{\bm{\epsilon}}+\partial_{z}^{2}q_{j}(\cdot,u_{j}^{\bm{\epsilon}})(\partial_{\epsilon_{1}}u_{j}^{\bm{\epsilon}})(\partial_{\epsilon_{2}}u_{j}^{\bm{\epsilon}})\\
 & \quad-\frac{\partial_{z}q_{j}(\cdot,u_{j}^{\bm{\epsilon}+\epsilon_{2}\bm{e}_{2}})\partial_{\epsilon_{1}}u_{j}^{\bm{\epsilon}+\epsilon_{2}\bm{e}_{2}}-\partial_{z}q_{j}(\cdot,u_{j}^{\bm{\epsilon}})\partial_{\epsilon_{1}}u_{j}^{\bm{\epsilon}}}{\epsilon_{2}}.
\end{align*}
After some computation we can write $
-\mathcal{G}_{2}=\mathcal{G}_{21}+\mathcal{G}_{22}+\mathcal{G}_{23},$
where
\begin{align*}
\begin{cases}
\mathcal{G}_{21}  =\partial_{z}q_{j}(\cdot,u_{j}^{\bm{\epsilon}})\big[v_{j}^{\bm{\epsilon}}-\delta_{\epsilon_{2}}\partial_{\epsilon_{1}}u_{j}^{\bm{\epsilon}}\big],\\
\mathcal{G}_{22}  =\big[\partial_{z}^{2}q_{j}(\cdot,u_{j}^{\bm{\epsilon}})(\partial_{\epsilon_{2}}u_{j}^{\bm{\epsilon}})-\frac{\partial_{z}q_{j}(\cdot,u_{j}^{\bm{\epsilon}+\epsilon_{2}\bm{e}_{2}})-\partial_{z}q_{j}(\cdot,u_{j}^{\bm{\epsilon}})}{\epsilon_{2}}\big](\partial_{\epsilon_{1}}u_{j}^{\bm{\epsilon}+\epsilon_{2}\bm{e}_{2}}),\\
\mathcal{G}_{23}  =\partial_{z}^{2}q_{j}(\cdot,u_{j}^{\bm{\epsilon}})\partial_{\epsilon_{2}}u_{j}^{\bm{\epsilon}}\big[\partial_{\epsilon_{1}}u_{j}^{\bm{\epsilon}}-\partial_{\epsilon_{1}}u_{j}^{\bm{\epsilon}+\epsilon_{2}\bm{e}_{2}}\big],
\end{cases}
\end{align*}
Note that 
$
\|v_{j}^{\bm{\epsilon}}-\delta_{\epsilon_{2}}\partial_{\epsilon_{1}}u_{j}^{\bm{\epsilon}}\|_{L^{\infty}(0,T;H^{s}(\mathbb{R}^{n}))\cap L^{\infty}(\mathbb{R}_{T}^{n})}\le C\|\mathcal{G}_{2}\|_{L^{\infty}(\Omega_{T})}.$
Possibly choosing a smaller $\epsilon_{0}$, we have 
$
\|\mathcal{G}_{21}\|_{L^{\infty}(\Omega_{T})}\le\frac{1}{2}\|v_{j}^{\bm{\epsilon}}-\delta_{\epsilon_{2}}\partial_{\epsilon_{1}}u_{j}^{\bm{\epsilon}}\|_{L^{\infty}(0,T;H^{s}(\mathbb{R}^{n}))\cap L^{\infty}(\mathbb{R}_{T}^{n})}.$Using the mean value theorem and Lemma~\ref{lem:first-derivative-well-defined},
we know that 
$
\lim_{\epsilon_{2}\rightarrow0} (\|\mathcal{G}_{22}\|_{L^{\infty}(\Omega_{T})} + \|\mathcal{G}_{23}\|_{L^{\infty}(\Omega_{T})}) = 0$.
We then conclude \eqref{eq:second-derivative-well-defined} by using arguments similar to Lemma~\ref{lem:first-derivative-well-defined}.
\end{proof}

Akin to Lemma~\ref{lem:information-first-linearized-DN-map}, we next demonstrate the following lemma.
 
\begin{lem}
\label{lem:information-second-linearized-DN-map}If $\Lambda_{q_{1}}(f)=\Lambda_{q_{1}}(f)$
for all $f\in\mathcal{C}_{c}^{\infty}(W_{T})$ with $\|f\|_{{\rm ext}}\le\tilde{\epsilon}_{0}$,
where $\tilde{\epsilon}_{0}$ is the constant given in Proposition~{\rm \ref{prop:well-posedness-diffusion-nonlinear}},
then there exists a constant $\epsilon_{0}=\epsilon_{0}(n,s,\Omega,T,\delta,\bm{g},m)>0$
with $0<\epsilon_{0}<\tilde{\epsilon}_{0}$ such that 
\begin{equation}
(-\Delta)^{s}\partial_{\epsilon_{1}\epsilon_{2}}u_{1}^{\bm{\epsilon}}\big|_{V_{T}}=(-\Delta)^{s}\partial_{\epsilon_{1}\epsilon_{2}}u_{2}^{\bm{\epsilon}}\big|_{V_{T}} \label{eq:information-second-linearized-DN-map}\quad \forall \; \bm{\epsilon}\; \mbox{with}\; |\bm{\epsilon}|\le\epsilon_{0}.
\end{equation}
 
\end{lem}

\begin{proof}
Using similar arguments as in Lemma~\ref{lem:information-first-linearized-DN-map}
(with Lemma~\ref{lem:second-derivative-well-defined}), we can show
that \eqref{eq:information-first-linearized-DN-map} implies \eqref{eq:information-second-linearized-DN-map}.
Then we conclude the lemma by Lemma~\ref{lem:information-first-linearized-DN-map}. 
\end{proof}
\begin{proof}
[Proof of Theorem~{\rm \ref{thm:main-diffusion}} for $m=2$] Using
\ref{itm:q3} and \eqref{eq:first-derivative-epsilon-zero}, we know that
$\partial_{\epsilon_{1}\epsilon_{2}}u_{j}^{\bm{\epsilon}}|_{\bm{\epsilon}=0}$
satisfies 
\[
\begin{cases}
\begin{array}{l}
(\partial_{t}+(-\Delta)^{s})(\partial_{\epsilon_{1}\epsilon_{2}}u_{j}^{\bm{\epsilon}}|_{\bm{\epsilon}=0})\\
\quad+\partial_{z}^{2}q_{j}(\cdot,0)(\partial_{\epsilon_{1}}u^{\bm{\epsilon}}|_{\bm{\epsilon}=0})(\partial_{\epsilon_{2}}u^{\bm{\epsilon}}|_{\bm{\epsilon}=0})=0
\end{array} & \text{in}\;\;\Omega_{T},\\
\partial_{\epsilon_{1}\epsilon_{2}}u_{j}^{\bm{\epsilon}}|_{\bm{\epsilon}=0}=0 & \text{in}\;\;\Omega_{T}^{e} \;\; \text{and on} \;\; \{0\} \times \mathbb R^n.
\end{cases}
\]
Hence, we know that $v:=\partial_{\epsilon_{1}\epsilon_{2}}u_{1}^{\bm{\epsilon}}|_{\bm{\epsilon}=0}-\partial_{\epsilon_{1}\epsilon_{2}}u_{2}^{\bm{\epsilon}}|_{\bm{\epsilon}=0}$
satisfies
\[
\begin{cases}
(\partial_{t}+(-\Delta)^{s})v+(\partial_{z}^{2}q_{1}(\cdot,0)-\partial_{z}^{2}q_{2}(\cdot,0))(\partial_{\epsilon_{1}}u^{\bm{\epsilon}}|_{\bm{\epsilon}=0})(\partial_{\epsilon_{2}}u^{\bm{\epsilon}}|_{\bm{\epsilon}=0})=0 \quad \text{in}\;\;\Omega_{T},\\
v=0 \qquad \text{in}\;\;\Omega_{T}^{e} \;\; \text{and on} \;\; \{0\} \times \mathbb R^n.
\end{cases}
\]
From Lemma~\ref{lem:information-second-linearized-DN-map}, we know
that 
\(
(-\Delta)^{s}v\big|_{V_{T}}=0.
\)
Since $v=0$ in $V_{T}$, using the unique continuation property of the
fractional Laplacian in Lemma~\ref{lem:UCP}, we conclude that $v\equiv0$.
Therefore, we know that 
\begin{equation}
\big( \partial_{z}^{2} q_{1}(\cdot,0) - \partial_{z}^{2} q_{2}(\cdot,0) \big) (\partial_{\epsilon_{1}} u^{\bm{\epsilon}} |_{\bm{\epsilon}=0}) (\partial_{\epsilon_{2}} u^{\bm{\epsilon}} |_{\bm{\epsilon}=0}) = 0. \label{eq:m2-to-recover}
\end{equation}
Since $g_{1},g_{2}\in\mathcal{C}_{c}^{\infty}(W_{T})$ are arbitrary,
using \eqref{eq:first-derivative-epsilon-zero-equation} and the Runge
approximation for fractional diffusion equation in Proposition~\ref{prop:Runge-diffusion},
we conclude $\partial_{z}^{2}q_{1}(\cdot,0)-\partial_{z}^{2}q_{2}(\cdot,0)=0$ in $\Omega_{T}$,
which proves Theorem~\ref{thm:main-diffusion} for $m=2$.
\end{proof}

\begin{proof}[Proof of Corollary~{\rm \ref{cor:diffusion-m-measurement}} for $m=2$] \label{idea:single-measurement}
Using the same argument as above, we reach \eqref{eq:m2-to-recover}: 
\begin{equation}
(\partial_{z}^{2}q_{1}(\cdot,0)-\partial_{z}^{2}q_{2}(\cdot,0))(\partial_{\epsilon_{1}}u^{\bm{\epsilon}}|_{\bm{\epsilon}=0})(\partial_{\epsilon_{2}}u^{\bm{\epsilon}}|_{\bm{\epsilon}=0})=0. \label{eq:m2-to-recover-recall}
\end{equation}
Using \eqref{eq:first-derivative-epsilon-zero-equation}, the unique continuation property of the fractional Laplacian in Lemma~\ref{lem:UCP} and a simple contradiction argument, for each $x_{0} \in \Omega$, we can find a sequence $\{x_{k}\}_{k \in \mathbb{N}} \subset \Omega$ with $x_{k} \rightarrow x_{0}$ and a sequence $\{t_{k}\}_{k \in \mathbb{N}} \subset (0,T)$ such that
\[
\partial_{\epsilon_{1}}u^{\bm{\epsilon}}|_{\bm{\epsilon}=0}(t_{k},x_{k}) \neq 0 \quad \text{and} \quad \partial_{\epsilon_{2}}u^{\bm{\epsilon}}|_{\bm{\epsilon}=0}(t_{k},x_{k}) \neq 0.
\]
Therefore from \eqref{eq:m2-to-recover-recall} we know that $\partial_{z}^{2}q_{1}(x_{k},0) = \partial_{z}^{2}q_{2}(x_{k},0)$ (here we have assumed that $\partial_{z}^{2}q_{1}(\cdot,0)$ and $\partial_{z}^{2}q_{2}(\cdot,0)$ are independent of $t$).
Hence by continuity of $\partial_{z}^{2}q_{1}(\cdot,0)$, $\partial_{z}^{2}q_{2}(\cdot,0)$ and the arbitrariness of $x_{0} \in \Omega$, we conclude 
$
	\partial_{z}^{2}q_{1}(\cdot,0)-\partial_{z}^{2}q_{2}(\cdot,0)=0\;\;\text{in}\;\;\Omega, 
$
which proves Corollary~\ref{cor:diffusion-m-measurement} for $m=2$.
\end{proof}

\subsection{Higher order linearization}

For each $2\le p\le m$, we denote $\partial_{(p)}=\partial_{\epsilon_{1}}\cdots\partial_{\epsilon_{p}}$.
By repeating formal differentiations to the equation \eqref{eq:formal-second-linearization},
we obtain the following $p$-th order linearization
\[
\begin{cases}
(\partial_{t}+(-\Delta)^{s})\partial_{(p)}u_{j}^{\bm{\epsilon}}+\partial_{(p)}q_{j}(\cdot,u_{j}^{\bm{\epsilon}})=0 & \text{in}\;\;\Omega_{T},\\
\partial_{(p)}u_{j}^{\bm{\epsilon}}=0 & \text{in}\;\;\Omega_{T}^{e} \;\; \text{and on} \;\; \{0\} \times \mathbb R^n.
\end{cases}
\]
where we simply denote $\partial_{(p)}=\partial_{\epsilon_{1}}\cdots\partial_{\epsilon_{p}}$.
By induction, we can verify 
\begin{align*}
\partial_{(p)}q_{j}(\cdot,u_{j}^{\bm{\epsilon}}) & =\partial_{z}q_{j}(\cdot,u_{j}^{\bm{\epsilon}})\partial_{(p)}u_{j}^{\bm{\epsilon}}+\sum_{\ell=2}^{p-1}\partial_{z}^{\ell}q_{j}(\cdot,u_{j}^{\bm{\epsilon}})\mathcal{T}_{p}^{\ell}(u_{j}^{\bm{\epsilon}})+\partial_{z}^{p}q_{j}(\cdot,u_{j}^{\bm{\epsilon}})\prod_{\ell=1}^{p}\partial_{\epsilon_{\ell}}u_{j}^{\bm{\epsilon}},
\end{align*}
where $\mathcal{T}_{p}^{\ell}(u_{j}^{\bm{\epsilon}})$ is a generic
notation (in order $p$ linearization) signifying a combination of
the terms $\partial_{\bm{\epsilon}}^{\bm{\alpha}}u_{j}^{\bm{\epsilon}}$
with multi-index $\bm{\alpha}$ satisfying $1\le|\bm{\alpha}|\le p-1$.
The following facts can be proved using strong induction on $m$: 
\begin{enumerate}
\item Functions $\partial_{(p)}u_{j}^{\bm{\epsilon}}\in L^{\infty}(0,T;H^{s}(\mathbb{R}^{n}))\cap L^{\infty}(\mathbb{R}_{T}^{n})$
are well-defined for each $2\le p\le m$.\smallskip 
\item There exists $\epsilon_{0}=\epsilon_{0}(n,s,\Omega,T,\delta,\bm{g},m)>0$
with $0<\epsilon_{0}<\tilde{\epsilon}_{0}$, where $\tilde{\epsilon}_{0}$
is the constant given in {\rm Proposition~\ref{prop:well-posedness-diffusion-nonlinear}},
such that 
\[
\lim_{\epsilon_{p}\rightarrow0}\|\partial_{(p)}u_{j}^{\bm{\epsilon}}-\delta_{\epsilon_{p}}\partial_{(p-1)}u_{j}^{\bm{\epsilon}}\|_{L^{\infty}(0,T;H^{s}(\mathbb{R}^{n}))\cap L^{\infty}(\mathbb{R}_{T}^{n})}=0\quad\text{for all}\;\;2\le p\le m,
\]
where 
$
\delta_{\epsilon_{p}}\partial_{(p-1)}u_{j}^{\bm{\epsilon}}=\frac{\partial_{(p-1)}u_{j}^{\bm{\epsilon}+\epsilon_{p}\bm{e}_{p}}-\partial_{(p-1)}u_{j}^{\bm{\epsilon}}}{\epsilon_{p}}\quad\text{in}\;\;\Omega_{T},
$
provided $|\bm{\epsilon}|+|\epsilon_{p}|<\epsilon_{0}$.\smallskip 
\item Moreover, if $\Lambda_{q_{1}}(f)=\Lambda_{q_{1}}(f)$ for all $f\in\mathcal{C}_{c}^{\infty}(W_{T})$
with $\|f\|_{{\rm ext}}\le\tilde{\epsilon}_{0}$, then we have 
\begin{equation}
(-\Delta)^{s}\partial_{(p)}u_{1}^{\bm{\epsilon}}\big|_{V_{T}}=(-\Delta)^{s}\partial_{(p)}u_{2}^{\bm{\epsilon}}\big|_{V_{T}}\quad\text{for all}\;\;2\le p\le m.\label{eq:linearized-DN}
\end{equation}
\end{enumerate}
Using the observations above, we are now ready to prove our main result. 
\begin{proof}
[Proof of Theorem~\ref{thm:main-diffusion}] We prove by induction on $m$. We assume the following hypothesis:
\begin{equation}
\partial_{z}^{p}q_{1}(\cdot,0)=\partial_{z}^{p}q_{1}(\cdot,0)\quad\text{for all}\;\;2\le p\le m-1.\label{eq:induction-hypothesis}
\end{equation}
Using \eqref{eq:first-derivative-epsilon-zero}, we see that 
\begin{align*}
\partial_{(m)}q_{j}(\cdot,u_{j}^{\bm{\epsilon}})|_{\bm{\epsilon}=0} & =\partial_{z}q_{j}(\cdot,0)\partial_{(m)}u_{j}^{\bm{\epsilon}}|_{\bm{\epsilon}=0}+\sum_{\ell=2}^{m-1}\partial_{z}^{\ell}q_{j}(\cdot,0)\mathcal{T}_{m}^{\ell}(u_{j}^{\bm{\epsilon}})\big|_{\bm{\epsilon}=0}+\partial_{z}^{m}q_{j}(\cdot,0)\prod_{\ell=1}^{m}\partial_{\epsilon_{\ell}}u^{\bm{\epsilon}}|_{\bm{\epsilon}=0}.
\end{align*}
Using \eqref{eq:induction-hypothesis}, we see that 
$
\sum_{\ell=2}^{m-1}\partial_{z}^{\ell}q_{j}(\cdot,0)\mathcal{T}_{m}^{\ell}(u_{1}^{\bm{\epsilon}}) \big|_{\bm{\epsilon}=0}
= \sum_{\ell=2}^{m-1}\partial_{z}^{\ell}q_{j}(\cdot,0)\mathcal{T}_{m}^{\ell}(u_{2}^{\bm{\epsilon}}) \big|_{\bm{\epsilon}=0}.$
Hence, we know that $v:=\partial_{(m)}u_{1}^{\bm{\epsilon}}|_{\bm{\epsilon}=0}-\partial_{(m)}u_{2}^{\bm{\epsilon}}|_{\bm{\epsilon}=0}$
satisfies 
\[
\begin{cases}
(\partial_{t}+(-\Delta)^{s})v+(\partial_{z}^{m}q_{1}(\cdot,0)-\partial_{z}^{m}q_{2}(\cdot,0))\prod_{\ell=1}^{m}\partial_{\epsilon_{\ell}}u^{\bm{\epsilon}}|_{\bm{\epsilon}=0}=0 & \text{in}\;\;\Omega_{T} \\
v=0 & \text{in}\;\;\Omega_{T}^{e} \;\; \text{and on} \;\; \{0\} \times \mathbb R^n.
\end{cases}
\]
Using \eqref{eq:linearized-DN}, we know that $
(-\Delta)^{s}v |_{V_{T}}=0.$
Since $v=0$ in $V_{T}$, by using the unique continuation principle of the fractional Laplacian (see Lemma~\ref{lem:UCP}), we conclude that $v\equiv0$.
Hence, we know that $(\partial_{z}^{m}q_{1}(\cdot,0)-\partial_{z}^{m}q_{2}(\cdot,0))\prod_{\ell=1}^{m}\partial_{\epsilon_{\ell}}u^{\bm{\epsilon}}|_{\bm{\epsilon}=0}=0\;\;\text{in}\;\;\Omega_{T}.$
Since $g_{1},\cdots,g_{m}\in\mathcal{C}_{c}^{\infty}(W_{T})$ are
arbitrary, using \eqref{eq:first-derivative-epsilon-zero-equation}
and the Runge approximation for the fractional diffusion equation proved in Proposition~\ref{prop:Runge-diffusion},
we conclude 
$\partial_{z}^{m}q_{1}(\cdot,0)=\partial_{z}^{m}q_{2}(\cdot,0)$ in $\Omega_{T}$.
This completes the proof of Theorem~\ref{thm:main-diffusion}. 
\end{proof}


\section{\label{sec:Analogous-result-for}Analogous result for the fractional
wave equation}

The following results can be proved by modifying the ideas in \cite[Corollary~2.2]{KLW21CalderonFractionalWave} (or \cite[Chapter~7]{Eva10PDE}),
which we  use later to prove the well-posedness of \eqref{eq:main-wave}
with small exterior data and to solve inverse problem as well. 

\begin{lem} \label{lem:wave-wellposedness}
Given any $n\in\mathbb{N}$ and $0<s<1$. Let $\Omega\subset\mathbb{R}^{n}$
be a bounded Lipschitz domain in $\mathbb{R}^{n}$, let $W\subset\Omega^{e}$
be any open set with Lipschitz boundary satisfying $\overline{W}\cap\overline{\Omega}=\emptyset$.
Let $a\in L^{\infty}(\Omega_{T})$. Then for any $F\in L^{2}(\Omega_{T})$,
$f\in\mathcal{C}_{c}^{\infty}(W_{T})$, $\psi\in\tilde{H}^{0}(\Omega)$,
$\varphi\in\tilde{H}^{s}(\Omega)$, there exists a unique solution
$u$ of 
\begin{equation}
\begin{cases}
(\partial_{t}^{2}+(-\Delta)^{s}+a)u=F & \text{in}\;\;\Omega_{T},\\
u=f & \text{in}\;\;\Omega_{T}^{e},\\
u=\varphi,\quad\partial_{t}u=\psi & \text{on}\;\;\{0\}\times\mathbb{R}^{n}.
\end{cases}\label{eq:linear-wave}
\end{equation}
satisfying 
\begin{align}
 & \|u-f\|_{L^{\infty}(0,T;H^{s}(\mathbb{R}^{n}))} + \|\partial_{t}u\|_{L^{\infty}(0,T;L^{2}(\Omega))} \nonumber \\
 & \le C (\|\varphi\|_{\tilde{H}^{s}(\Omega)} + \|\psi\|_{L^{2}(\Omega)} + \|F-(-\Delta)^{s}f\|_{L^{2}(\Omega_{T})} )\label{eq:wave-regularity}
\end{align}
for some constant $C=C(n,s,T,\|a\|_{L^{\infty}(\Omega_{T})})$. 
\end{lem}

\begin{rem}
It is interesting to compare \eqref{eq:wave-regularity} with \eqref{eq:diffusion-regularity}: both solutions (wave and diffusion) have regularity $L^{\infty}(0,T;H^{s}(\mathbb{R}^{n}))$. In \cite[Corollary~2.2]{KLW21CalderonFractionalWave}, they only consider the case when $a$ is independent of time $t$.
The existence of solutions can be proved using exactly the same argument, but the proof of uniqueness result need extra care. In contrast to \cite[Chapter~7]{Eva10PDE}, here we do not assume the $W^{1,\infty}(0,T;L^{\infty}(\Omega))$ regularity for the coefficient $a$. 
\end{rem}

\begin{proof}[Proof of uniqueness result of Lemma~{\rm \ref{lem:wave-wellposedness}}]
Let $u\in L^{2}(0,T;H^{s}(\mathbb{R}^{n}))\cap H^{1}(0,T;L^{2}(\Omega))$ be the solution of 
\begin{equation} \label{eq:patch1-uniqueness1}
\begin{cases}
(\partial_{t}^{2}+(-\Delta)^{s}+a)u=0 & \text{in }\Omega_{T},\\
u=0 & \text{in }\Omega_{T}^{e},\\
u=\partial_{t}u=0 & \text{on }\{0\}\times\mathbb{R}^{n}.
\end{cases}
\end{equation}
We want to show that $u\equiv0$. Fix $0\le\eta\le T$ and set
\[
v(t,\cdot):=\begin{cases}
\int_{t}^{\eta}u(\tau,\cdot)\,\mathsf{d}\tau & \text{if }0\le t\le\eta,\\
0 & \text{if }\eta\le t\le T.
\end{cases}
\]
Then $v(t,\cdot)\in\tilde{H}^{s}(\Omega)$ for each $0\le t\le T$, and so 
\begin{equation} \label{eq:patch1-uniqueness2}
\int_{\Omega}\int_{0}^{\eta}(\partial_{t}^{2}u)v \,\mathsf{d}t \,\mathsf{d}x + \int_{\mathbb{R}^{n}}\int_{0}^{\eta}(-\Delta)^{\frac{s}{2}}u(-\Delta)^{\frac{s}{2}}v \,\mathsf{d}t \,\mathsf{d}x + \int_{\Omega}\int_{0}^{\eta}auv \,\mathsf{d}t \,\mathsf{d}x = 0.
\end{equation}
Since $\partial_{t}u(0)=v(\eta)=0$ and $\partial_{t}v=-u$ for all $0\le t\le\eta$, we see that 
\begin{align}
\int_{\Omega}\int_{0}^{\eta}(\partial_{t}^{2}u)v \,\mathsf{d}t \,\mathsf{d}x
& =-\int_{\Omega}\int_{0}^{\eta}(\partial_{t}u)(\partial_{t}v) \,\mathsf{d}t \,\mathsf{d}x
= \int_{\Omega}\int_{0}^{\eta}(\partial_{t}u)u \,\mathsf{d}t \,\mathsf{d}x \nonumber \\
& = \frac{1}{2} \frac{d}{dt} \int_{\Omega}\int_{0}^{\eta}|u|^{2} \,\mathsf{d}t \,\mathsf{d}x
= \frac{1}{2} \|u(\eta,\cdot)\|_{L^{2}(\Omega)}^{2}. \label{eq:patch1-uniqueness3}
\end{align}
Using the fact $\partial_{t}v=-u$ for all $0\le t\le\eta$, we also have
\begin{align}
\int_{\mathbb{R}^{n}}\int_{0}^{\eta}(-\Delta)^{\frac{s}{2}}u & (-\Delta)^{\frac{s}{2}}v \,\mathsf{d}t \,\mathsf{d}x
= -\int_{\mathbb{R}^{n}}\int_{0}^{\eta}\partial_{t} \big( (-\Delta)^{\frac{s}{2}}v \big) (-\Delta)^{\frac{s}{2}}v \,\mathsf{d}t \,\mathsf{d}x \nonumber \\
& =-\frac{1}{2}\int_{\mathbb{R}^{n}}\int_{0}^{\eta}\frac{d}{dt}|(-\Delta)^{\frac{s}{2}}v|^{2} \,\mathsf{d}t \,\mathsf{d}x
=\frac{1}{2}\|(-\Delta)^{\frac{s}{2}}v(0,\cdot)\|_{L^{2}(\mathbb{R}^{n})}^{2}.\label{eq:patch1-uniqueness4}
\end{align}
Since $a\in L^{\infty}(\Omega_{T})$, combining \eqref{eq:patch1-uniqueness2}, \eqref{eq:patch1-uniqueness3} and \eqref{eq:patch1-uniqueness4}, together with the Hardy-Littlewood-Sobolev inequality \eqref{eq:Hardy-Littlewood-Sobolev}, we obtain 
\begin{equation} \label{eq:patch1-uniquenss5}
\|v(0,\cdot)\|_{\tilde{H}^{s}(\Omega)}^{2}+\|u(\eta,\cdot)\|_{L^{2}(\Omega)}^{2}\le C\int_{0}^{\eta}\big(\|v(t,\cdot)\|_{L^{2}(\Omega)}^{2}+\|u(t,\cdot)\|_{L^{2}(\Omega)}^{2}\big) \,\mathsf{d}t.
\end{equation}
Let us write $w(t,\cdot):=\int_{0}^{t}u(\tau,\cdot)\,\mathsf{d} \tau$.
Since $v(0,\cdot)=w(\eta,\cdot)$ and $v(t,\cdot)=w(\eta,\cdot)-w(t,\cdot)$, from \eqref{eq:patch1-uniquenss5} we know that
\begin{align*}
\|w(\eta,\cdot)\|_{\tilde{H}^{s}(\Omega)}^{2}+\|u(\eta,\cdot)\|_{L^{2}(\Omega)}^{2}
& \lesssim \int_{0}^{\eta}\big(\|w(\eta,\cdot)-w(t,\cdot)\|_{L^{2}(\Omega)}^{2}+\|u(t,\cdot)\|_{L^{2}(\Omega)}^{2}\big) \,\mathsf{d}t \\
& \lesssim \eta \|w(\eta,\cdot)\|_{L^2(\Omega)} + \int_{0}^{\eta} \big( \|w(t,\cdot)\|_{L^{2}(\Omega)}^{2} + \|u(t,\cdot)\|_{L^{2}(\Omega)}^{2} \big) \,\mathsf{d}t.
\end{align*}
Therefore, we can choose $T_{1}$, which is independent of $\eta$, such that 
\begin{equation*}
\|w(\eta,\cdot)\|_{\tilde{H}^{s}(\Omega)}^{2}+\|u(\eta,\cdot)\|_{L^{2}(\Omega)}^{2}
\leq C\int_{0}^{\eta}\big(\|w(t,\cdot)\|_{L^{2}(\Omega)}^{2}+\|u(t,\cdot)\|_{L^{2}(\Omega)}^{2}\big) \,\mathsf{d}t,
\end{equation*}
for all $0 \leq \eta \leq T_1$.
Using the Gr\"{o}nwall's inequality in \cite[Section~B.2]{Eva10PDE}, we know that $u(t,\cdot)=0$ for all $t\in[0,T_{1}]$. Applying the same argument on the intervals $[T_{1},2T_{1}]$, $[2T_{1},3T_{1}]$, etc., we conclude that $u\equiv0$. 
\end{proof}

We need the following Sobolev embedding to obtain $L^{\infty}(\Omega_{T})$-regularity of the solution, which is a special case of \cite[Theorem~8.2]{DNPV12FractionalSobolev}: 
\begin{lem}
[\cite{DNPV12FractionalSobolev}] \label{lem:embedding}Let $n = 1$ and $1/2<s<1$.
There exists a constant $C=C(s,\Omega)$ such that 
$
\|f\|_{\mathcal{C}^{0,\alpha}(\Omega)}\le C (\|f\|_{L^{2}(\Omega)}^{2}+[f]_{\dot{H}^{s}(\Omega)}^{2} )
\le C\|f\|_{H^{s}(\mathbb{R})}^{2},
$
for any $f\in L^{2}(\Omega)$ with $\alpha = (2s-1)/2$. 
\end{lem}

Therefore, Lemma~\ref{lem:wave-wellposedness} implies the following
result.
 
\begin{prop}
\label{prop:well-posed-suitable-regularity-wave}
Let $n=1$ and $1/2<s<1$.
Let $\Omega\subset\mathbb{R}$ be a bounded open set in $\mathbb{R}^{n}$,
let $W\subset\Omega^{e}$ be any open set satisfying $\overline{W}\cap\overline{\Omega}=\emptyset$. Let $a\in L^{\infty}(\Omega_{T})$.
Then for any $\tilde{F}\in L^{\infty}(\Omega_{T})$ and $f\in\mathcal{C}_{c}^{\infty}(W_{T})$,
there exists a unique weak solution $u$ of \eqref{eq:linear-wave}
satisfying 
\begin{align*}
& \quad\|u\|_{L^{\infty}(0,T;H^{s}(\mathbb{R}^{1})) \cap L^{\infty}(\mathbb{R}_{T}^{1})} + \|\partial_{t}u\|_{L^{\infty}(0,T;L^{2}(\Omega))} \\
& \leq C\big(\|\varphi\|_{\tilde{H}^{s}(\Omega)} + \|\psi\|_{L^{2}(\Omega)} + \|F\|_{L^{2}(\Omega_{T})} \\
& \quad + \|f\|_{L^{\infty}(0,T;H^{s}(\mathbb{R}^{1}))\cap L^{\infty}(\mathbb{R}_{T}^{1})} + \|(-\Delta)^{s}f\|_{L^{2}(\Omega_{T})} \big)
\end{align*}
for certain constant $C=C(s,T,\|a\|_{L^{\infty}(\Omega_{T})},\Omega)$. 
\end{prop}

Using the same argument as in Proposition~\ref{prop:well-posedness-diffusion-nonlinear},
we also can prove the well-posedness of \eqref{eq:main-wave} for small
exterior data:

\begin{prop} \label{prop:well-posedness-wave-nonlinear}
Let $n=1$ and $\frac{1}{2}<s<1$.
Let $\Omega\subset\mathbb{R}$ be a bounded open set,
let $W\subset\Omega^{e}$ be any open set satisfying $\overline{W}\cap\overline{\Omega}=\emptyset$.
Fix any parameter $\delta>0$. Assume $q$ satisfies {\rm \ref{itm:q1}}--{\rm \ref{itm:q3}}.
There exists a sufficiently small parameter $\tilde{\epsilon}_{0}=\tilde{\epsilon}_{0}(s,\Omega,T,\delta)>0$
such that the following statement holds: Given any $f\in\mathcal{C}_{c}^{\infty}(W_{T})$
with $\|f\|_{{\rm ext}} \le \tilde{\epsilon}_{0}$, there exists a unique solution $u\in L^{\infty}(0,T;H^{s}(\mathbb{R}^{1}))\cap L^{\infty}(\mathbb{R}_{T}^{1})$
of \eqref{eq:main-wave} with 
\begin{equation}
\|u\|_{L^{\infty}(0,T;H^{s}(\mathbb{R}^{1}))\cap L^{\infty}(\mathbb{R}_{T}^{1})}\le C\|f\|_{{\rm ext}}\label{eq:continuity-dependence-solution-1}
\end{equation}
for certain constant $C=C(s,T,\Omega)$. 
\end{prop}

Finally, the inverse problem for the nonlinear fractional wave equation
\eqref{eq:main-wave}, i.e.~Theorem~\ref{thm:main-wave} and Corollary~\ref{cor:wave-m-measurement}, can be proved
using exactly the same idea as in Theorem~\ref{thm:main-diffusion} and Corollary~\ref{cor:diffusion-m-measurement}, respectively, see Section \ref{sec:inverse-problem-diffusion}.

\appendix

\section{Well-posedness of the linear fractional
diffusion equation} \label{sec:well-posed-linear}

\subsection{Uniqueness of weak solution}

We first prove the uniqueness of weak solution of \eqref{eq:linear-diffusion1} as well as \eqref{eq:linear-diffusion2}.
It suffices to prove the following statement: If $u$ a weak solution of 
\begin{equation} \label{eq:linear-diffusion3}
\begin{cases}
(\partial_{t}+(-\Delta)^{s}+a)v=0 & \text{in}\;\;\Omega_{T},\\
v=0 & \text{in}\;\;\Omega_{T}^{e} \;\; \text{and on} \;\; \{0\} \times \mathbb R^n.
\end{cases}
\end{equation}
then $u\equiv0$. Multiplying the first equation of \eqref{eq:linear-diffusion3}
by $v$, we obtain 
\begin{align*}
0
& = \langle\bm{v}',\bm{v}\rangle+\mathcal{B}[\bm{v},\bm{v};t]
= \frac{\mathsf{d}}{\mathsf{d}t}\big(\frac{1}{2}\|\bm{v}(t)\|_{L^{2}(\Omega)}^{2}\big)+\mathcal{B}[\bm{v},\bm{v};t] \\
& \ge\frac{\mathsf{d}}{\mathsf{d}t}\big(\frac{1}{2}\|\bm{v}(t)\|_{L^{2}(\Omega)}^{2}\big)-\|a\|_{L^{\infty}(\Omega_{T})}\|\bm{v}(t)\|_{L^{2}(\Omega)}^{2},
\end{align*}
that is, $
\frac{\mathsf{d}}{\mathsf{d}t}\big(\|\bm{v}(t)\|_{L^{2}(\Omega)}^{2}\big)\le2\|a\|_{L^{\infty}(\Omega_{T})}\|\bm{v}(t)\|_{L^{2}(\Omega)}^{2}.$
Using the Gr\"onwall's inequality in \cite[Section~B.2]{Eva10PDE},
we conclude $\|\bm{v}(t)\|_{L^{2}(\Omega)}^{2}=0$ for all $0\le t\le T$, hence $u\equiv0$.
The uniqueness is proved.

\subsection{Existence of weak solution}

Now it suffices prove that there exists a weak solution of \eqref{eq:linear-diffusion2}. 

\textbf{Step 1: Galerkin approximation.} We now set up the Galerkin
approximation for \eqref{eq:linear-diffusion2}. Similar to \cite[Appendix~A]{KLW21CalderonFractionalWave},
we consider an eigenbasis $\{w_{k}\}_{k\in\mathbb{N}}$ associated
with the Dirichlet fractional Laplacian in a bounded domain $\Omega$.
We normalize these eigenfunctions so that 
\begin{align*}
\{w_{k}\}_{k\in\mathbb{N}} & \text{ be an orthogonal basis in }\tilde{H}^{s}(\Omega),\\
\{w_{k}\}_{k\in\mathbb{N}} & \text{ be an orthonormal basis in }L^{2}(\Omega).
\end{align*}
Given any fixed integer $m\in\mathbb{N}$, we consider the following
ansatz: 
\begin{equation}
\bm{v}_{m}(t):=\sum_{k=1}^{m}d_{m}^{k}(t)w_{k}.\label{eq:0ansatz}
\end{equation}
Plugging the ansatz \eqref{eq:0ansatz} into Definition~\ref{def:weak-solution}(b), we obtain
\begin{equation}
\begin{cases}
(\bm{v}_{m}'(t),w_{k})_{L^{2}(\Omega)}+\mathcal{B}[\bm{v}_{m},w_{k};t]=(\tilde{\bm{F}}(t),w_{k})_{L^{2}(\Omega)} & \text{for all}\;\;0\le t\le T,\\
d_{m}^{k}(0)=(\tilde{\varphi},w_{k})_{L^{2}(\Omega)}.
\end{cases}\label{eq:0Galerkin1}
\end{equation}
Note that 
$(\bm{v}_{m}',w_{k})_{L^{2}(\Omega)}  =(d_{m}^{k})'(t)$, $\mathcal{B}[\bm{v}_{m},w_{k};t] =\sum_{\ell=1}^{m}e^{k\ell}(t)d_{m}^{k}(t)$ with the coefficients $e^{k\ell}(t):=\mathcal{B}[w_{\ell},w_{k};t].
$
This shows that $d_{m}^{k}(t)$ satisfies the following linear system
of ordinary differential equation (ODE): 
\[
\begin{cases}
(d_{m}^{k})'(t)+\sum_{\ell=1}^{m}e^{k\ell}(t)d_{m}^{k}(t)=(\tilde{\bm{F}}(t),w_{k})_{L^{2}(\Omega)} & \text{for all}\;\;0\le t\le T,\\
d_{m}^{k}(0)=(\tilde{\varphi},w_{k})_{L^{2}(\Omega)}.
\end{cases}
\]
Therefore, the standard ODE theory guarantees the existence and uniqueness
of such $d_{m}^{k}(t)$, and thus \eqref{eq:0ansatz} is a valid discretization
of \eqref{eq:linear-diffusion2}. 

\textbf{Step 2: Energy estimate.} Multiplying \eqref{eq:0Galerkin1}
by $d_{m}^{k}(t)$, and summing over index $k=1,\cdots,m$, we have
\begin{equation} \label{eq:0Galerkin-energy1}
(\bm{v}_{m}',\bm{v}_{m})_{L^{2}(\Omega)}+\mathcal{B}[\bm{v}_{m},\bm{v}_{m};t]=(\tilde{\bm{F}},\bm{v}_{m})_{L^{2}(\Omega)}.
\end{equation}
The following Hardy-Littlewood-Sobolev inequality can be found in \cite[Prop.~15.5]{Pon16elliptic} or in
\cite[equation~(A.11)]{KLW21CalderonFractionalWave}:
\begin{equation}
\|\bm{v}_{m}\|_{L^{2}(\mathbb{R}^{1})}=\|\bm{v}_{m}\|_{L^{2}(\Omega)}\le C(n,s)\|\phi\|_{L^{\frac{2n}{n-s}}(\mathbb{R}^{n})}\le C(n,s)\|(-\Delta)^{\frac{s}{2}}\phi\|_{L^{2}(\mathbb{R}^{1})} \label{eq:Hardy-Littlewood-Sobolev}
\end{equation}
for $n = 1$ and for all $\phi \in \tilde{H}^{s}(\Omega)$. On the other hand, we observe that 
$
(\bm{v}_{m}',\bm{v}_{m})_{L^{2}(\Omega)}=\frac{\mathsf{d}}{\mathsf{d}t}\big(\frac{1}{2}\|\bm{v}_{m}\|_{L^{2}(\Omega)}^{2}\big).$
Hence, from \eqref{eq:0Galerkin-energy1} we have 
\begin{equation}
\frac{\mathsf{d}}{\mathsf{d}t}\big(\frac{1}{2}\|\bm{v}_{m}\|_{L^{2}(\Omega)}^{2}\big)+\|\bm{v}_{m}\|_{\tilde{H}^{s}(\Omega)}^{2}\le C(n,s,\|a\|_{L^{\infty}(\Omega_{T})})\big(\|\bm{v}_{m}\|_{L^{2}(\Omega)}^{2}+\|\tilde{\bm{F}}\|_{L^{2}(\Omega)}^{2}\big)\label{eq:0Galerkin-energy2}
\end{equation}
for all $ 0\le t\le T$.
Using the Gr\"onwall's inequality in \cite[Section~B.2]{Eva10PDE},
we have 
\[
\|\bm{v}_{m}(t)\|_{L^{2}(\Omega)}^{2}\le e^{Ct} \big(\|\bm{v}_{m}(0)\|_{L^{2}(\Omega)}^{2}+C\int_{0}^{t}\|\tilde{\bm{F}}(s)\|_{L^{2}(\Omega)}^{2}\,\mathsf{d}s\big)\quad\text{for all}\;\;0\le t\le T.
\]
Since $
\|\bm{v}_{m}(0)\|_{L^{2}(\Omega)}^{2}=\sum_{k=1}^{m}|(\tilde{\varphi},w_{k})_{L^{2}(\Omega)}|^{2}\le\sum_{k=1}^{\infty}|(\tilde{\varphi},w_{k})_{L^{2}(\Omega)}|^{2}=\|\varphi\|_{L^{2}(\Omega)}^{2},$
then we have 
\begin{equation}
\sup_{0\le t\le T}\|\bm{v}_{m}(t)\|_{L^{2}(\Omega)}^{2}\le C_{s,T,\|a\|_{\infty}}\big(\|\varphi\|_{L^{2}(\Omega)}^{2}+\|\tilde{\bm{F}}\|_{L^{2}(\Omega_{T})}^{2}\big).\label{eq:0Galerkin-energy3}
\end{equation}
Integrating \eqref{eq:0Galerkin-energy2} on $t\in[0,T]$, we obtain
\begin{equation}
\|\bm{v}_{m}\|_{L^{2}(0,T;\tilde{H}^{s}(\Omega))}^{2}\le C_{s,\|a\|_{\infty}}\big(\|\bm{v}_{m}\|_{L^{2}(\Omega_{T})}^{2}+\|\tilde{\bm{F}}\|_{L^{2}(\Omega_{T})}^{2}\big).\label{eq:0Galerkin-energy4}
\end{equation}
Combining \eqref{eq:0Galerkin-energy3} and \eqref{eq:0Galerkin-energy4},
we obtain the following energy estimate: 
\begin{align}
 & \quad\sup_{0\le t\le T}\|\bm{v}_{m}(t)\|_{L^{2}(\Omega)}^{2}+\|\bm{v}_{m}\|_{L^{2}(0,T;\tilde{H}^{s}(\Omega))}^{2}\le C(n,s,T,\|a\|_{L^{\infty}(\Omega_{T})}) (\|\varphi\|_{L^{2}(\Omega)}^{2}+\|\tilde{\bm{F}}\|_{L^{2}(\Omega_{T})}^{2} ).\label{eq:0Galerkin-energy5}
\end{align}
Fixing any $\phi\in\tilde{H}^{s}(\Omega)$ with $\|\phi\|_{\tilde{H}^{s}(\Omega)}\le1$,
we write $\phi=\phi_{1}+\phi_{2}$, where $\phi_{1}\in{\rm span}\,\{w_{k}\}_{k=1}^{m}$
and $(\phi_{2},w_{k})_{L^{2}(\Omega)}=0$ for $k=1,\cdots,m$. Using
\eqref{eq:0Galerkin1}, we see that 
\[
(\bm{v}_{m}'(t),\phi)_{L^{2}(\Omega)}=(\bm{v}_{m}'(t),\phi_{1})_{L^{2}(\Omega)}=(\tilde{\bm{F}},\phi_{1})-\mathcal{B}[\bm{v}_{m},\phi_{1};t].
\]
Since $\|\phi_{1}\|_{\tilde{H}^{s}(\Omega)}\le1$, this implies 
$
|(\bm{v}_{m}'(t),\phi)_{L^{2}(\Omega)}|\le C (\|\tilde{\bm{F}}(t)\|_{L^{2}(\Omega)}^{2}+\|\bm{v}_{m}\|_{\tilde{H}^{s}(\Omega)}^{2} )$.
Hence we know that
\[
\|\bm{v}_{m}'(t)\|_{H^{-s}(\Omega)}^{2}
:= \sup_{\|\phi\|_{\tilde{H}^{s}(\Omega)}\le1}|(\bm{v}_{m}'(t),\phi)_{L^{2}(\Omega)}|\le C_{\|a\|_{\infty}} (\|\tilde{\bm{F}}(t)\|_{L^{2}(\Omega)}^{2}+\|\bm{v}_{m}\|_{\tilde{H}^{s}(\Omega)}^{2}).
\]
Integrating the inequality above on $t\in[0,T]$, and combining the result with
\eqref{eq:0Galerkin-energy5}, we obtain 
\begin{align}
 & \sup_{0\le t\le T}\|\bm{v}_{m}(t)\|_{L^{2}(\Omega)}^{2}+\|\bm{v}_{m}\|_{L^{2}(0,T;\tilde{H}^{s}(\Omega))}^{2}+\|\bm{v}_{m}'\|_{L^{2}(0,T;H^{-s}(\Omega))}^{2}\nonumber \\
 & \quad\le C_{s,T,\|a\|_{\infty}} (\|\varphi\|_{L^{2}(\Omega)}^{2}+\|\tilde{\bm{F}}\|_{L^{2}(\Omega_{T})}^{2}).\label{eq:0Galerkin-energy6}
\end{align}

\textbf{Step 3: Passing to the limit.}
By \eqref{eq:0Galerkin-energy6},
we can extract a subsequence of $\{\bm{v}_{m}\}_{m\in\mathbb{N}}$,
 still denoted by  $\{\bm{v}_{m}\}_{m\in\mathbb{N}}$ (for simplicity),
such that 
\begin{equation}
\begin{cases}
\bm{v}_{m}\rightharpoonup\bm{v} & \text{weakly in }L^{2}(0,T;\tilde{H}^{s}(\Omega)),\\
\bm{v}_{m}'\rightharpoonup\bm{v}' & \text{weakly in }L^{2}(0,T;H^{-s}(\Omega)).
\end{cases}\label{eq:0weak-limit}
\end{equation}
Given any fixed integer $N$, we write $
\tilde{\bm{v}}(t):=\sum_{k=1}^{N}d^{k}(t)w_{k},$
where $d^{k}(t)~(k=1, \cdots, N)$ are arbitrary smooth functions (not the one in \eqref{eq:0ansatz}).
Choosing $m\ge N$, multiplying \eqref{eq:0Galerkin1} by $d^{k}(t)$, and summing over $k=1,\cdots,N$, we obtain 
\begin{equation}
\int_{0}^{T}\big((\bm{v}_{m}'(t),\tilde{\bm{v}}(t))_{L^{2}(\Omega)}+\mathcal{B}[\bm{v}_{m},\tilde{\bm{v}};t]\big)\,\mathsf{d}t=\int_{0}^{T}(\tilde{\bm{F}}(t),\tilde{\bm{v}}(t))_{L^{2}(\Omega)}\,\mathsf{d}t.\label{eq:0weak-limit-before}
\end{equation}
Taking $m \to +\infty$ in \eqref{eq:0weak-limit-before}, and from \eqref{eq:0weak-limit}, we know that
\begin{equation}
\int_{0}^{T}\big(\langle\bm{v}'(t),\tilde{\bm{v}}(t)\rangle+\mathcal{B}[\bm{v},\tilde{\bm{v}};t]\big)\,\mathsf{d}t=\int_{0}^{T}(\tilde{\bm{F}}(t),\tilde{\bm{v}}(t))_{L^{2}(\Omega)}\,\mathsf{d}t.\label{eq:0weak-limit-after}
\end{equation}
Due to the arbitrariness of $N$ and $\{d^{k}\}_{k=1}^{N}$, we have
\[
\langle\bm{v}',\phi\rangle+\mathcal{B}[\bm{v},\phi;t]=(\tilde{\bm{F}}(t),\phi)_{L^{2}(\Omega)}\quad\text{for all}\;\;\phi\in\tilde{H}^{s}(\Omega).
\]
This together with \eqref{eq:0Galerkin-energy6}
verifies Definition~\ref{def:weak-solution}(a)(b). 

It remains to show $\bm{v}$ verifies Definition~\ref{def:weak-solution}(c).
To that end, let us choose any $\tilde{\bm{v}}\in\mathcal{C}^{1}(0,T;\tilde{H}^{s}(\Omega))$
with $\tilde{\bm{v}}(T)=0$. From \eqref{eq:0weak-limit-after}, we
have 
\begin{equation}
\int_{0}^{T}\big((\tilde{\bm{v}}'(t),\bm{v}(t))_{L^{2}(\Omega)}+\mathcal{B}[\bm{v},\tilde{\bm{v}};t]\big)\,\mathsf{d}t=\int_{0}^{T}(\tilde{\bm{F}}(t),\tilde{\bm{v}}(t))_{L^{2}(\Omega)}\,\mathsf{d}t+(\tilde{\bm{v}}'(0),\bm{v}(0))_{L^{2}(\Omega)}.\label{eq:0weak-limit-after1}
\end{equation}
Similarly, from \eqref{eq:0weak-limit-before}, we have 
\begin{equation}
\int_{0}^{T}\big((\tilde{\bm{v}}'(t),\bm{v}_{m}(t))_{L^{2}(\Omega)}+\mathcal{B}[\bm{v}_{m},\tilde{\bm{v}};t]\big)\,\mathsf{d}t=\int_{0}^{T}(\tilde{\bm{F}}(t),\tilde{\bm{v}}(t))_{L^{2}(\Omega)}\,\mathsf{d}t+(\tilde{\bm{v}}'(0),\bm{v}_{m}(0))_{L^{2}(\Omega)}.\label{eq:0weak-limit-before1}
\end{equation}
Combining \eqref{eq:0weak-limit} and \eqref{eq:0weak-limit-before1},
we obtain 
\[
\int_{0}^{T} \big((\tilde{\bm{v}}'(t),\bm{v}(t))_{L^{2}(\Omega)}+\mathcal{B}[\bm{v},\tilde{\bm{v}};t]\big)\,\mathsf{d}t=\int_{0}^{T}(\tilde{\bm{F}}(t),\tilde{\bm{v}}(t))_{L^{2}(\Omega)}\,\mathsf{d}t+(\tilde{\bm{v}}'(0),\varphi)_{L^{2}(\Omega)}.
\]
Comparing this with \eqref{eq:0weak-limit-after1}, we see that
$
(\tilde{\bm{v}}'(0),\bm{v}(0))_{L^{2}(\Omega)}=(\tilde{\bm{v}}'(0),\varphi)_{L^{2}(\Omega)}.$
Due to the arbitrariness of $\tilde{\bm{v}}$, we conclude that $\bm{v}$
verifies Definition~\ref{def:weak-solution}(c). 

\textbf{Step 4. Higher regularity.} We now further assume $\varphi\in\tilde{H}^{s}(\Omega)$.
Multiplying \eqref{eq:0Galerkin1} by $(d_{m}^{k})'(t)$, and summing
over $k=1,\cdots,m$, we have 
\begin{equation}
(\bm{v}_{m}',\bm{v}_{m}')_{L^{2}(\Omega)}+\mathcal{B}[\bm{v}_{m},\bm{v}_{m}']=(\tilde{\bm{F}},\bm{v}_{m}')_{L^{2}(\Omega)}.\label{eq:0Galerkin-energy-improve1}
\end{equation}
Note that  we have
\begin{align*}
\mathcal{B}[\bm{v}_{m},\bm{v}_{m}']
 & = \int_{\mathbb{R}^{n}}(-\Delta)^{\frac{s}{2}}\bm{v}_{m}(t)(-\Delta)^{\frac{s}{2}}\bm{v}_{m}'(t)\,\mathsf{d}x + \int_\Omega a(t,x) \bm{v}_{m}(t,x) \bm{v}_{m}'(t,x) \,\mathsf{d}x \\
 & =\frac{\mathsf{d}}{\mathsf{d}t}\big(\frac{1}{2}\int_{\mathbb{R}^{1}}|(-\Delta)^{\frac{s}{2}}\bm{v}_{m}(t)|^{2} \,\mathsf{d}x \big) + \int_\Omega a(t,x) \bm{v}_{m}(t,x) \bm{v}_{m}'(t,x) \,\mathsf{d}x,
\end{align*}
and
$\int_{\Omega}a(t,x)\bm{v}_{m}(t)\bm{v}_{m}'(t)\,\mathsf{d}x
\le \epsilon\|\bm{v}_{m}'(t)\|_{L^{2}(\Omega)}^{2}+C\epsilon^{-1}\|\bm{v}_{m}(t)\|_{L^{2}(\Omega)}^{2}$
and 
$
|(\tilde{\bm{F}},\bm{v}_{m}')_{L^{2}(\Omega)}|\le\epsilon\|\bm{v}_{m}'(t)\|_{L^{2}(\Omega)}^{2}+C\epsilon^{-1}\|\tilde{\bm{F}}(t)\|_{L^{2}(\Omega)}^{2}.$ These along with  \eqref{eq:0Galerkin-energy-improve1} imply
\begin{align*}
 & \|\bm{v}_{m}'(t)\|_{L^{2}(\Omega)}^{2}+\frac{\mathsf{d}}{\mathsf{d}t}\big(\frac{1}{2}\int_{\mathbb{R}^{n}}|(-\Delta)^{\frac{s}{2}}\bm{v}_{m}(t)|^{2}\,\mathsf{d}x\big)\\
 & \le2\epsilon\|\bm{v}_{m}'(t)\|_{L^{2}(\Omega)}^{2}+C\epsilon^{-1}\|\bm{v}_{m}(t)\|_{L^{2}(\Omega)}^{2}+C\epsilon^{-1}\|\tilde{\bm{F}}(t)\|_{L^{2}(\Omega)}^{2}.
\end{align*}
Choosing $\epsilon=1/4$, we obtain 
\begin{equation} \label{eq:0Galerkin-energy-improve2}
\|\bm{v}_{m}'(t)\|_{L^{2}(\Omega)}^{2}+\frac{\mathsf{d}}{\mathsf{d}t}\big(\int_{\mathbb{R}^{n}}|(-\Delta)^{\frac{s}{2}}\bm{v}_{m}(t)|^{2}\,\mathsf{d}x\big)
\le C(\|\bm{v}_{m}(t)\|_{L^{2}(\Omega)}^{2}+\|\tilde{\bm{F}}(t)\|_{L^{2}(\Omega)}^{2}).
\end{equation}
Given any $0\le\tilde{t}\le T$, we integrate \eqref{eq:0Galerkin-energy-improve2}
on $t\in[0,\tilde{t}]$,
\begin{align*}
 & \int_{0}^{\tilde{t}}\|\bm{v}_{m}'(t)\|_{L^{2}(\Omega)}^{2}\,\mathsf{d}t+\int_{\mathbb{R}^{n}}|(-\Delta)^{\frac{s}{2}}\bm{v}_{m}(\tilde{t})|^{2}\,\mathsf{d}x-\int_{\mathbb{R}^{n}}|(-\Delta)^{\frac{s}{2}}\bm{v}_{m}(0)|^{2}\,\mathsf{d}x\\
 & \le C\big(\int_{0}^{\tilde{t}}\|\bm{v}_{m}(t)\|_{L^{2}(\Omega)}^{2}+\int_{0}^{\tilde{t}}\|\tilde{\bm{F}}(t)\|_{L^{2}(\Omega)}^{2}\big)
  \le C\big(\|\bm{v}_{m}\|_{L^{2}(\Omega_{T})}^{2}+\|\tilde{F}\|_{L^{2}(\Omega_{T})}^{2}\big).
\end{align*}
Combining this inequality with \eqref{eq:Hardy-Littlewood-Sobolev},
we obtain 
\begin{align}
  \|\bm{v}_{m}'\|_{L^{2}(\Omega_{T})}^{2}+\|\bm{v}_{m}\|_{L^{\infty}(0,T;\tilde{H}^{s}(\Omega))}^{2}
 &\le C\big(\|\bm{v}_{m}'\|_{L^{2}(\Omega_{T})}^{2}+\sup_{0\le\tilde{t}\le T}\int_{\mathbb{R}^{n}}|(-\Delta)^{\frac{s}{2}}\bm{v}_{m}(\tilde{t})|^{2}\,\mathsf{d}x\big)\nonumber \\
 & \le C\big(\int_{\mathbb{R}^{n}}|(-\Delta)^{\frac{s}{2}}\bm{v}_{m}(0)|^{2}\,\mathsf{d}x+\|\bm{v}_{m}\|_{L^{2}(\Omega_{T})}^{2}+\|\tilde{F}\|_{L^{2}(\Omega_{T})}^{2}\big)\nonumber \\
 & \le C(\|\bm{v}_{m}(0)\|_{\tilde{H}^{s}(\Omega)}^{2}+\|\bm{v}_{m}\|_{L^{2}(\Omega_{T})}^{2}+\|\tilde{F}\|_{L^{2}(\Omega_{T})}^{2}\big).\label{eq:0Galerkin-energy-improve3}
\end{align}
Since $\|\bm{v}_{m}(0)\|_{\tilde{H}^{s}(\Omega)}^2 \leq \sum_{k=1}^{\infty} |(g,w_{k})_{L^{2}(\Omega)}|^{2}\|w_{k}\|_{\tilde{H}^{s}(\Omega)}^2
= \|\varphi\|_{\tilde{H}^{s}(\Omega)}^{2}$, \eqref{eq:0Galerkin-energy-improve3} implies 
\begin{equation*}
	 \|\bm{v}_{m}'\|_{L^{2}(\Omega_{T})}^{2}+\|\bm{v}_{m}\|_{L^{\infty}(0,T;\tilde{H}^{s}(\Omega))}^{2}
	 \le C(\|\tilde{\varphi}\|_{\tilde{H}^{s}(\Omega)}^{2}+\|\bm{v}_{m}\|_{L^{2}(\Omega_{T})}^{2}+\|\tilde{F}\|_{L^{2}(\Omega_{T})}^{2}).
\end{equation*}
Therefore, combining this inequality with \eqref{eq:0Galerkin-energy6},
we obtain 
\begin{equation*}
	\|\bm{v}_{m}'\|_{L^{2}(\Omega_{T})}^{2}+\|\bm{v}_{m}\|_{L^{\infty}(0,T;\tilde{H}^{s}(\Omega))}^{2}
	\le C(\|\varphi\|_{\tilde{H}^{s}(\Omega)}^{2}+\|\tilde{F}\|_{L^{2}(\Omega_{T})}^{2}).
\end{equation*}
Finally, taking the limit $m\rightarrow\infty$, we complete our proof.

\section{Some discussions} \label{sec:conclusion}

The main difficulty in proving Theorem \ref{thm:main-wave} is the regularity of the solutions. 
Due to this difficulty, we are only able to prove Theorem~\ref{thm:main-wave} in one dimension.
The method we used requires the $L^{\infty}(\Omega_{T})$-regularity for the linear fractional wave equation.
The $L^{\infty}(\Omega_{T})$-regularity is required to guarantee the well-posedness of \eqref{eq:main-wave},
and it is essential to prove that the linearization is well-defined as we see in Section \ref{sec:The-forward-problems}.
However, we are only able to obtain this regularity in the case when $n=1$ and $\frac 1 2 < s < 1$.
If one can prove the well-posedness of \eqref{eq:main-wave} for general $n \in \mathbb{N}$ and $0<s<1$, then Theorem~\ref{thm:main-wave} immediately extends for general $n \in \mathbb{N}$ and $0<s<1$.

In view of standard elliptic regularity results (see \cite{GT01Elliptic} or \cite[Proposition~A.1]{JLS17pharmonic} in terms of other norms) as well as Sobolev embedding, an attempt to improve the result in Theorem~\ref{thm:main-wave} is to try to obtain the $L^{\infty}(0,T;H^{2s}(\mathbb{R}^{n}))$ regularity for the solution. However, this idea is less likely to be feasible. 
Using \cite[Lemma~2.3]{GSU20Calderon}, we know there exists a unique solution $w\in\tilde{H}^{s}(\Omega)$ of 
\begin{equation} \label{eq:fractional-Dirichlet}
	\left\{\begin{aligned}
		(-\Delta)^{s} w & = F && \text{in}\;\;\Omega,\\
		w& =0 && \text{in}\;\;\Omega^{e},
	\end{aligned}\right.
\end{equation}
for $F\in H^{-s}(\Omega)$.
Choose $\Omega$ to be the unit disk and $F$ to be a positive constant in $\Omega$.
Then the best regularity result of \eqref{eq:fractional-Dirichlet} we know is $\mathcal{C}^{s}(\mathbb{R}^{n})$ \cite[Proposition~7.2]{Ros16FractionalLaplacianSurvey}.
In fact, \cite[Lemma~5.4]{Ros16FractionalLaplacianSurvey} gives an explicit solution $w(x) = (1-|x|^{2})_{+}^{s} \in \mathcal{C}^{s}(\mathbb{R}^{n})$, and such $w$ does not belong to $\mathcal{C}^{s'}(\mathbb{R}^{n})$ for any $s' > s$.
When $n=1$, we have the continuous embedding $H^{2s}(\mathbb{R}) \hookrightarrow C^{2s - \frac{1}{2}}(\mathbb{R})$. Therefore, at least when $n=1$ and $s > \frac{1}{2}$, such a solution $w$ of \eqref{eq:fractional-Dirichlet} cannot be in $H^{2s}(\mathbb{R})$.


\section*{Acknowledgements}
All the authors were partly supported by the Academy of Finland (Centre of Excellence in Inverse Modelling and Imaging, grant 284715) and by the European Research Council under Horizon 2020 (ERC CoG 770924).


{


}

\end{document}